\documentclass[11pt]{amsart}
\usepackage{amssymb,mathrsfs,graphicx,enumerate}

\title[   ]{$L^2$-contraction of large planar shock waves for multi-dimensional scalar viscous conservation laws}

\author[Kang]{Moon-Jin Kang}
\address[Moon-Jin Kang]{
\newline Department of Mathematics, \newline The University of Texas at Austin, Austin, TX 78712, USA}
\email{moonjinkang@math.utexas.edu}

\author[Vasseur]{Alexis F. Vasseur}
\address[Alexis F. Vasseur]{\newline Department of Mathematics, \newline The University of Texas at Austin, Austin, TX 78712, USA}
\email{vasseur@math.utexas.edu}

\author[Wang]{Yi Wang}
\address[Yi Wang]{\newline Institute of Applied Mathematics, AMSS,  CAS, Beijing 100190, China
 and Beijing Center of Mathematics
and Information Sciences, Beijing 100048, P. R. China}
\email{wangyi@amss.ac.cn}

\def\v{\varepsilon}

\def\f{\frac}

\def\di{\displaystyle}
\def\i{\infty}
 \newtheorem{lemma}{\bf Lemma}[section]
       \newtheorem{theorem}[lemma]{\bf Theorem}

       \newtheorem{proposition}[lemma]{\bf Proposition}
       \newtheorem{remark}[lemma]{\bf Remark}

\newcommand{\bbr}{\mathbb R}

\newcommand{\bbt} {\mathbb T}

\numberwithin{figure}{section}
\newcommand{\beq}{\begin{equation}}
\newcommand{\eeq}{\end{equation}}
\newcommand{\bsp}{\begin{split}}
\newcommand{\esp}{\end{split}}

\def\eps{\varepsilon }

\newcommand\adots{\mathinner{\mkern2mu\raise1pt\hbox{.}
\mkern3mu\raise4pt\hbox{.}\mkern1mu\raise7pt\hbox{.}}}

\def\charf {\mbox{{\text 1}\kern-.30em {\text l}}}

\def\XXint#1#2#3{{\setbox0=\hbox{$#1{#2#3}{\int}$}
\vcenter{\hbox{$#2#3$}}\kern-.5\wd0}}

\def\XXparallel#1#2#3{{\setbox0=\hbox{$#1{#2#3}{\parallel}$}
\vcenter{\hbox{$#2#3$}}\kern-.6\wd0}}

\begin{document}

\date{\today}


\thanks{\textbf{Acknowledgment.} M.-J. Kang was partially supported by Basic Science Research Program through the National Research Foundation of Korea (NRF-2013R1A6A3A03020506).
  A. F. Vasseur was partially supported by the NSF Grant DMS 1209420. Y. Wang is supported by NSFC grant No. 11322106, Youth Innovation Promotion Association of CAS and Young top-notch talent Program of Organization Department of CCCPC.
}

\begin{abstract}
 We consider a $L^2$-contraction of large viscous shock waves for the multi-dimensional scalar viscous conservation laws, up to a suitable shift. The shift function depends on the time and space variables. It solves a parabolic equation with inhomogeneous coefficients reflecting the perturbation. We consider a suitably small $L^2$-perturbation around a viscous planar shock wave of arbitrarily large strength. However, we do not impose any condition on the anti-derivative variables of the perturbation around shock profile. More precisely, it is proved that if the initial perturbation around the viscous shock wave is suitably small in the $L^2$ norm, then the $L^2$-contraction holds true for the viscous shock wave up to a shift function which may depend on the temporal and spatial variables. Moreover, as the time $t$ tends to infinity, the $L^2$-contraction holds true up to a time-dependent shift function. In particular, if we choose some special initial perturbation, then we can prove a $L^2$ convergence of the solutions towards the associated shock profile up to a time-dependent shift.
\end{abstract}
\maketitle \centerline{\date}

\section{Introduction and Main results}
\setcounter{equation}{0}

We consider the multi-dimensional scalar viscous conservation laws
\begin{equation}\label{eq}
\left\{
\begin{array}{ll}
\di \partial_tu+{\rm div} A(u)= \Delta u,\\
\di u(t=0,x)=u_0(x),
\end{array}
\right.
\end{equation}
where $t\in \mathbb{R}^+$, $x=(x_1,x^\prime)\in \bbr\times\mathbb{T}^{N-1}$ with $\mathbb{T}^{N-1}$ being $N-1$ dimensional torus, $ N\geq 2$, $u=u(t,x)\in \mathbb{R}$, and $A(u)=(A_1(u), A_2(u),\cdots, A_N(u))^t\in\mathbb{R}^N$ is a smooth vector field of $N$ fluxes $A_i$, with $A_1$ being strictly convex, i.e., $A_1''(u)>0$, $\forall~ u\in \mathbb{R}$.\\
 Without loss of generality, we consider stationary planar shock waves $U(x_1)$ satisfying
\begin{equation}\label{shock}
\left\{
\begin{array}{ll}
\di (A_{1}(U))'= U'', \\[3mm]
\di U(x_1)\rightarrow u_\pm,\quad {\rm as} ~~x_1\rightarrow\pm \infty,
\end{array}
\right.
\end{equation}
where $x_1 \in \bbr$ denotes the normal direction, and $x'$ the transverse directions parallel to the shock front. Here, the two end points $u_{\pm}$ satisfy $u_->u_+$ by the strict convexity of $A_1$ and the Lax entropy condition, and $A_1(u_+)=A_1(u_-)$.  The existence of the stationary shock profile to \eqref{shock} is well-known and the profile is unique up to a constant shift (see for example \cite{IO}).

In this article, we consider a $L^2$-contraction of large shock waves $U(x_1)$ in \eqref{shock} for the multi-dimensional scalar viscous conservation laws \eqref{eq}. There are many literatures concerning the stability of viscous shock wave to the viscous conservation laws in one-dimensional case. In 1960s, Il'in-Oleinik \cite{IO} first proved the time-asymptotic stability of viscous shock waves to the scalar equation \eqref{eq} when $N=1$. Then, Goodman \cite{G} and Matsumura-Nishihara \cite{MN} independently proved the stability of viscous shock waves to the system case under the zero mass condition on the perturbation about the shock profile. Then, by introducing suitable constant shift on the shock profile and the linear and nonlinear diffusion waves in the transverse characteristic fields, Liu \cite{Liu} removed the zero mass condition in \cite{G, MN}. Furthermore, Spzessy-Xin \cite{SX} introduced the coupled diffusion waves to improve the stability result in \cite{Liu}.  Recently, Vasseur-Yao \cite{VY} removed the smallness on the shock strength in \cite{MN} by introducing a new entropy variable. For the multi-dimensional case $N\ge 2$, Kruzhkov \cite{K} first proved the $L^1$ contraction for the multidimensional scalar viscous conservation laws \eqref{eq}, using Kruzhkov entropies. Goodman \cite{G1} proved the stability of weak shocks based on the anti-derivative variables by introducing the shift function depending on the spatial and temporal variables. Hoff-Zumbrun \cite{HZ, HZ1} improved the stability result in \cite{G1} to the large shock waves. Notice that the above stability results are all based on the energy methods or point-wise Green function methods by using the anti-derivative variables to the perturbation around the shock profile. 
On the other hand, Freist$\ddot{\mbox{u}}$hler-Serre \cite{FS} proved the large-time $L^1$ stability of large perturbations of viscous shocks to scalar conservation laws \eqref{eq} when $N=1$. 

Another method for the $L^2$-type stability is based on the relative entropy method, which is purely nonlinear, and allows to handle rough and large perturbations. The relative entropy method was first introduced by Dafermos \cite{Da} and Diperna \cite{Di} to prove the $L^2$ stability and uniqueness of Lipschitzian solutions to the hyperbolic conservation laws endowed with a convex entropy. In \cite{Di}, that was also used to get uniqueness of some discontinuous solutions in some particular cases. However, no stability result was obtained in this paper. Later, Chen-Frid \cite{CF2, CF3} and Chen-Frid-Li \cite{CFL} used this method to prove the uniqueness and asymptotic  stability of Riemann solutions to some hyperbolic conservation laws. The theory of stability of discontinuous solutions, based on the relative entropy has been reformulated in \cite{SV, KV2} in terms of contraction, up to a shift. Recently, the method was used by  Leger in \cite{L} to show the $L^2$-contraction up to a shift of inviscid shocks to the scalar conservation laws (see also \cite{A} for an extension  to $L^p$, $1<p<\infty$).  That has been extended to the system case in \cite{LV} for extreme shocks, and general criteria have been developed in \cite{KV3}, \cite{SV} for possibly all shocks including intermediate characteristic fields. 
The relative entropy method is also an effective method for the study of  asymptotic limits. One of the first usage of the method in this context is due to Yau \cite{Yau} for the hydrodynamic limit of Ginzburg-Landau models. Since then, there have been many works  in this context, see \cite{BGL1, BGL2, BTV, BV, GS, KV1, LM, MS} etc. and the survey paper \cite{V}, although they are all considering the limit to a smooth (Lipschitz) limit function. Recently, the relative entropy method has been successfully applied to showing the vanishing viscosity limit of the viscous scalar conservation laws to shocks \cite{CV}, and the zero dissipation limit of full compressible Navier-Stokes-Fourier system to contact discontinuities \cite{VW}.
Furthermore, that has been also successfully used to prove the $L^2-$contraction of viscous shock profiles to the one-dimensional scalar viscous conservation laws \cite{KV}, up to a time-dependent shift. 

The present paper is the first attempt to use the relative entropy method to study the $L^2$ contraction of viscous planar shock waves to the multidimensional viscous conservation laws. Unlike the one-dimensional case in \cite{KV2}, there is a more difficult issue for the multi-dimensional case since the perturbation may propagate along the transverse directions. More precisely, we need to define a spatially inhomogeneous shift function, for which we have the contraction of the viscous shock. The main difficulty is to prove the global-in-time existence of the shift function. On the other hand, if we choose a special initial perturbation, then we have that the special perturbation is contractive and time-asymptotically converges to the viscous shock wave up to the time-dependent shift. Our results require the initial perturbations to be suitably small in $L^2(\bbr\times\bbt^{N-1})$ but the shock strength can be arbitrarily large.

For notational convenience, we will denote the spatial domain by 
$$
\Omega:=\bbr\times \bbt^{N-1}.
$$
Our first result is the following.
\begin{theorem}\label{thm:general}
Let $U$ be a planar shock wave defined by \eqref{shock}. Then, for any fixed $t_0>0$, there exist a positive constant $\delta_0$ and a shift function $Y(t,x)$ such that, for any initial data $u_0$ with $\|u_0-U\|_{L^2(\Omega)}<\delta_0$ and $u_0\in L^{\infty}(\Omega)$, the solution $u$ to \eqref{eq} with the initial data $u_0$ satisfies that $\int_{\Omega} |u(t,x)-U(x_1+ Y(t,x))|^2 dx$ is non-increasing in time for $t>t_0$. Moreover, there exists a positive constant $C(t_0)$ depending on $t_0$ such that
\beq\label{thm-claim1}
\int_{\Omega} |u(t,x)-U(x_1+ Y(t,x))|^2 dx \le C(t_0)\int_{\Omega} |u_0(x)-U(x_1)|^2 dx,\quad \forall t>0.
\eeq
The spatially inhomogeneous shift $Y(t,x)$ can be constructed such that
\begin{align}
\begin{aligned} \label{Y-final}
&\|\sqrt{|U'(\cdot+m(t))|} (Y-m(t))\|_{L^{\infty}(0,\infty; L^2(\Omega))}+\|\sqrt{|U'(\cdot+m(t))|} \nabla Y\|_{L^2((0,\infty)\times\Omega)}\le C\delta_0,\\
& \|\nabla  Y\|_{L^{\infty}(0,\infty; L^2(\Omega))}+ \|\Delta  Y\|_{L^2((0,\infty)\times\Omega)}\le C\delta_0,\\
&\|\nabla  Y\|_{L^{\infty}(0,\infty; H^s_{loc}(\Omega))}+ \|\Delta  Y\|_{L^2(0,\infty; H^{s}_{loc}(\Omega))}\le C(t_0)\delta_0,
\end{aligned}
\end{align}
where $s>\frac{N}{2}$, and $C$ is some positive constant.\\
Furthermore, we have the following time-asymptotic behavior for the shift $Y$:
\beq\label{Y-final2}
\lim_{t\to\infty}\int_{\Omega}|U( x_1+Y(t,x))-U(x_1+m(t))|^2 dx=0,
\eeq
where the spatially homogeneous shift $m(t)$ satisfies
\begin{equation}\label{mt}
m(t)=\frac{\di \int_\Omega |U^\prime(x_1+m(t))|Ydx}{\di \int_\Omega |U^\prime(x_1+m(t))|dx}.
\end{equation}
\end{theorem}

\begin{remark}\label{remark-v}
In proof of Theorem \ref{thm:general}, we will consider the shift $Y$ as a solution of a parabolic equation
\begin{equation}\label{Y-eq}
\left\{
\begin{array}{ll}
\di\partial_t Y -A'_1(U( Y+x_1))\partial_{x_1} Y +\sum_{i=2}^{N}A'_i(U( Y+x_1))\partial_{x_i}Y\\
\di\quad -A'_1(U(Y+x_1))|\nabla_x Y|^2+w\cdot \nabla_x Y-\Delta Y \\
\di \quad  =- (w_{1}- h_M(t)) \psi_M (x_1+m(t)) - h_M(t) - g(t),\\
\di Y|_{t=0}=0.
\end{array}\right.
\end{equation}
Here, $w=(w_1,\cdots,w_N)$ is a vector field defined by 
\beq\label{w}
w=\varphi(t)\frac{A(u|U(Y+x_1))}{u-U(Y+x_1)},
\eeq
where $\varphi$ is a smooth function such that $0\le\varphi\le 1$ and
\[
\varphi(t)=\left\{ \begin{array}{ll}
          0& \mbox{if $ 0<t<\frac{t_0}{2}$},\\
        1& \mbox{if $t>t_0$},\end{array} \right.
\]  
where $t_0$ is the arbitrarily fixed constant in Theorem \ref{thm:general}.\\
Moreover, $h_M$ is an average of $w_1$ as
\[
h_M(t):=\frac{1}{2(M+1)}\int_{\bbt^{N-1}}\int_{|x_1+m(t)|\leq M+1} w_1 dx,
\]
where $m(t)$ is defined by \eqref{mt}, and $M$ is some constant sufficiently large, and
$$
g(t)=\int_{\Omega} (u- U(Y+x_1))U^\prime(x_1+m(t)) dx.
$$
\end{remark}

\begin{remark}\label{remark-m}
In Theorem \ref{thm:general}, the smallness condition on $u_0-U$ is only in $L^2(\Omega)$. In addition to Theorem \ref{thm:general}, we will show that if there exists a constant $\delta_0>0$ such that $\|u_0-U\|_{H^s(\Omega)}<\delta_0$ for $s>\frac{N}{2}$, then the solution $u$ to \eqref{eq} with the initial data $u_0$ satisfies the $L^2$-contraction for all $t>0$, i.e.,
\beq\label{thm-claim2}
\int_{\Omega} |u(t,x)-U(x_1+ Y(t,x))|^2 dx \le \int_{\Omega} |u_0(x)-U(x_1)|^2 dx,\quad t>0,
\eeq
where the shift $Y(t,x)$ can be constructed as a solution of the above equation \eqref{Y-eq} without $\varphi$, i.e., $\varphi(t)= 1$ for all $t>0$, thus the shift $Y$ satisfies the above properties \eqref{Y-final} and \eqref{Y-final2}. 
As a consequence, we have a time-asymptotic $L^2$-contraction of the shock up to the spatially homogenous shift $m(t)$, i.e.,
\[
\lim_{t\to\infty}\int_{\Omega}|u(t,x)-U(x_1+m(t))|^2 dx\le \int_{\Omega}|u_0(x)-U(x_1)|^2 dx.
\]
\end{remark}


\begin{remark}\label{mt-1}
Notice that since $ \int_\Omega |U^\prime(x_1+m(t))|dx= \int_\Omega |U^\prime(x_1)|dx$, it follows from \eqref{mt} that
$$
\begin{array}{ll}
\di m^\prime(t)=\frac{\partial_t\int_\Omega|U^\prime(x_1+m(t))|Y(t,x_1,x^\prime) dx}{\int_\Omega |U^\prime(x_1)|dx}=\frac{\partial_t\int_\Omega|U^\prime(x_1)|Y(t,x_1-m(t),x^\prime) dx}{\int_\Omega |U^\prime(x_1)|dx}\\[4mm]
\di \qquad =\frac{\int_\Omega|U^\prime(x_1)|\big(Y_t (t,x_1-m(t),x^\prime) -m^\prime(t)\partial_{x_1}Y(t,x_1-m(t),x^\prime)\big)dx}{\int_\Omega |U^\prime(x_1)|dx}.
\end{array}
$$
Therefore, $m(t)$ satisfies the following ODE
\begin{equation}\label{mt2}
\begin{array}{ll}
\di m^\prime(t) \int_\Omega|U^\prime(x_1)|\big(1+\partial_{x_1}Y(t,x_1-m(t),x^\prime)\big) dx=\int_\Omega |U^\prime(x_1+m(t))|Y_t (t,x_1,x^\prime)dx\\
\di  =\int_\Omega |U^\prime(x_1+m(t))|\Big[ A'_1(U(  Y+x_1))\partial_{x_1} Y -\sum_{i=2}^{N}A'_i(U( Y+x_1))\partial_{x_i}Y+A'_1(U(Y+x_1))|\nabla_x Y|^2\\
\di\quad\qquad\qquad  -w\cdot \nabla_x Y+\Delta Y - (w_{1}- h_M(t)) \psi_M (x_1+m(t)) - h_M(t) - g(t)\Big] dx.
\end{array}
\end{equation}
with the initial value
$$
m(0)=0.
$$
Thanks to the smallness condition on $\|\nabla Y\|_{L^{\infty}((0,\infty)\times\Omega)}$ in \eqref{Y-final}, the ODE \eqref{mt2} on $m(t)$ has a unique global-in-time solution.
\end{remark}

\begin{remark}
1. Theorem \ref{thm:general} holds true for arbitrarily large shock wave and any spatial dimension $N\geq2$. Moreover, we only assume that the $L^2$-perturbation $u_0-U$ is suitably small, while the oscillations of the solution, BV-norm of the solution can be arbitrarily large.\\
2. We do not impose any conditions on the anti-derivative variables on the perturbation of shock, which is quite different from the previous results in \cite{G, HZ, HZ1}.
\end{remark}

Our second result is on a special kind of perturbation:
\begin{theorem}\label{special} 
Let $u_0=U(x_1+Y_0(x))$ for any $Y_0\in L^{\infty}(\Omega)$ with $\|Y_0\|_{L^{\infty}(\Omega)}<\delta_0$ for some small constant $\delta_0>0$. Then there exists $Y\in L^{\infty}((0,\infty)\times\Omega)$ such that the solution of \eqref{eq} satisfies $u(t,x)=U(x_1+Y(t,x))$ and $Y$ satisfies that
\beq\label{Y-reg}
\|\sqrt{|U'(x_1)|} (Y-c(t))\|_{L^{\infty}(0,\infty; L^2(\Omega))}+\|\sqrt{|U'(x_1)|} \nabla Y\|_{L^2((0,\infty)\times\Omega)}\le C\delta_0,
\eeq
where $c(t)$ is defined by 
\begin{equation}\label{ct}
c(t)=\frac{\int_\Omega |U^\prime(x_1)|Y(t,x) dx}{\int_\Omega |U^\prime(x_1)|dx}.
\end{equation}
Furthermore, the perturbation $u=U(x_1+Y(t,x))$ time-asymptotically converges towards the shock wave $U$ up to a time-dependent shift $c(t)$, i.e.,
\[
\lim_{t\to\infty}\int_{\Omega}|u(t,x)-U(x_1+c(t))|^2 dx=0.
\]
\end{theorem}

\begin{remark}
1. For Theorem \ref{special}, we will construct the shift $Y$ as a solution of a parabolic equation
\begin{equation}\label{Y-eq-0}
\left\{
\begin{array}{ll}
\di\partial_t Y -A'_1(U(Y+x_1))\partial_{x_1}Y +\sum_{i=2}^{N}A'_i(U( Y+x_1))\partial_{x_i}Y -A'_1(U( Y+x_1))|\nabla_x Y|^2-\Delta Y = 0,\\
\di Y|_{t=0}=Y_0.
\end{array}\right.
\end{equation}
2. Notice that since $U'\in L^2(\Omega)$ by \eqref{U-prop}, we easily see that the specific perturbation $u_0(x)=U(x_1+Y_0(x))$ is a small $L^2$-perturbation of the shock profile $U$. 
\end{remark}

The paper is organized as follows. In the next section, we derive an energy equality based on the relative entropy method, and present basic properties of the shock waves and useful inequalities, which are crucial for our analysis. We will first prove Theorem \ref{special} in Section 3. Its proof is simpler than the one of Theorem \ref{thm:general}. It is worthwhile to present first the main ideas in this context. Section 4 is dedicated to the proofs of Theorem \ref{thm:general} and the claim in Remark \ref{remark-m}. We first prove the claim of Remark \ref{remark-m} and then Theorem \ref{thm:general}. In Appendix, we present a proof on local-in-time existence of the shift as a solution to \eqref{Y-eq}.

\section{Preliminaries}
\setcounter{equation}{0}
In this section, we present an energy equality based on the relative entropy method, and basic properties on the viscous shock waves, and then useful inequalities, which are needed for our analysis in the following sections. 
\subsection{Relative entropy method}
In this part, we present a useful energy equality based on the relative entropy method as follows.
\begin{lemma}\label{lem-entropy}
Let $u$ be the smooth solution of the conservation laws \eqref{eq}, and $V$ be a smooth solution of a nonlinear parabolic equation
\begin{equation}\label{VG}
V_t+{\rm div} A(V)-\Delta V+w\cdot \nabla V=G,
\end{equation}
where $w$ and $G$ are some inhomogeneous coefficient functions.
Then, we have
\begin{align}
\begin{aligned}\label{ineq-0} 
&\frac{1}{2}\frac{d}{dt}\int_{\Omega} |u-V|^2 dx + \int_{\Omega} |\nabla(u-V)|^2 dx \\
&\quad =  -\int_{\Omega}\Big(A(u|V) - (u-V) w\Big)\cdot \nabla V dx + \int_{\Omega} (u-V)Gdx.
\end{aligned}
\end{align}
\end{lemma}

The remaining part is devoted to the proof of Lemma \ref{lem-entropy}. Even though our framework is based on the $L^2$-norm, we here present the general case of the relative entropy $\eta(\cdot|\cdot)$ for a given entropy $\eta$. Then, we will focus on the quadratic entropy and explain why the choice of quadratic entropy is essential. Concerning the following relative entropy method, we refer to \cite{KV2}. \\

For a strictly convex entropy $\eta$ of the scalar conservation laws $\eqref{eq}$, we define the associated relative entropy function by
\[
\eta(u|v)=\eta(u)-\eta(v) -\eta^{\prime}(v) (u-v),
\]
and the relative flux by 
\[
A(u|v):=A(u)-A(v)-A^{\prime}(v)(u-v).
\]
Let $q(\cdot,\cdot)$ be the flux of the relative entropy defined by
\[
q(u,v) = q(u)-q(v) -\eta^{\prime} (v) (A(u)-A(v)),
\]
where $q$ is the entropy flux of $\eta$, i.e., $q^{\prime} = \eta^{\prime} A^{\prime}$.\\
We now investigate the relative entropy between the solution $u$ of \eqref{eq} and the solution $V$ of \eqref{VG}.
 A straightforward computation together with \eqref{eq} and \eqref{VG} yields that
\begin{align*}
\begin{aligned} 
\partial_t \eta(u|V)  &= (\eta^{\prime}(u) - \eta^{\prime}(V)) \partial_t u - \eta^{\prime\prime}(V) (u-V) \partial_t V\\
& = \underbrace{-(\eta^{\prime}(u) - \eta^{\prime}(V)) {\rm div} A(u) +  \eta^{\prime\prime}(V) (u-V){\rm div} A(V)}_{I}+\eta^{\prime\prime}(V) (u-V) w\cdot\nabla V \\
&\quad+(\eta^{\prime}(u) - \eta^{\prime}(V))\Delta u -   \eta^{\prime\prime}(V) (u-V) \Delta V +\eta^{\prime\prime}(V) (u-V)G.
\end{aligned}
\end{align*} 
Since the flux part $I$ above can be written by
\begin{align*}
\begin{aligned} 
I &= - {\rm div} q(u,V) -\eta''(V)A(u|V)\cdot \nabla V,
\end{aligned}
\end{align*} 
we have
\begin{align*}
\begin{aligned} 
\partial_t \eta(u|V)  &= - {\rm div} q(u,V) +(\eta^{\prime}(u) - \eta^{\prime}(V))\Delta u -   \eta^{\prime\prime}(V) (u-V) \Delta V
 \\
 &\quad -\eta''(V)A(u|V)\cdot \nabla V +\eta^{\prime\prime}(V) (u-V) w\cdot\nabla V  +\eta^{\prime\prime}(V) (u-V)G.
\end{aligned}
\end{align*} 
Then, we integrate the above equality over $\Omega$ to get
\begin{align*}
\begin{aligned}
&\frac{d}{dt}  \int_{\Omega} \eta(u|V) dx \\
&\quad= \int_{\Omega}\Big((\eta^{\prime}(u) - \eta^{\prime}(V))\Delta u -   \eta^{\prime\prime}(V) (u-V) \Delta V \Big) dx\\
& \qquad+ \int_{\Omega}\Big(-\eta''(V)A(u|V)\cdot \nabla V +\eta^{\prime\prime}(V) (u-V) w\cdot\nabla V\Big)dx+ \int_{\Omega}\eta^{\prime\prime}(V)(u-V)Gdx.
\end{aligned}
\end{align*}
Now, if we consider the quadratic entropy $\eta(u)=\frac{u^2}{2}$, then the parabolic term induces a positive dissipation. Therefore, we have \eqref{ineq-0}.

\subsection{Properties of viscous shock wave $U$}
We briefly present some well-known properties of shock profile $U$, which are crucially used in the proofs of main results.  
We first mention that the shock profile $U$ exponentially converges towards the two end points $u_{\pm}$.  Since $A''_1>0$, it follows from \eqref{shock} that $U$ satisfies the compressibility condition
\beq\label{U-sign}
U^{\prime}  <0,
\eeq
and the R-H condition $A_1(u_+)=A_1(u_-)$
and the Lax entropy condition $A_1^{\prime} (u_+)<0<A_1^{\prime} (u_-)$ hold true.  Thus, there exist positive constants $c_{\pm}$ such that
\beq\label{U-prop}
|U'(x_1)| \sim  \exp(-c_{\pm} |x_1|)\quad\mbox{as} ~x_1\rightarrow \pm \infty.
\eeq
Indeed, since
\[
\frac{A_1(U)-A_1(u_{\pm}) }{U-u_{\pm}}\rightarrow A^{\prime} (u_{\pm}) \quad \mbox{as} ~U\rightarrow u_{\pm},
\]
it follows from \eqref{shock} that 
\begin{align*}
\begin{aligned}
U'=A_1(U)-A_1(u_{\pm}) \sim A_1'(u_{\pm})(U-u_{\pm} ) \quad \mbox{as} ~U\rightarrow u_{\pm},
\end{aligned}
\end{align*}
which together with the above Lax condition implies \eqref{U-prop}.\\
In addition, by the Lax entropy condition, there exists a unique state $u_*\in(u_+, u_-)$ such that 
$$
A_1^\prime(u_*)=0.
$$
Let $U(x_{1*})=u_*$, then it is worth noticing that the monotonicity condition \eqref{U-sign} together with $A_1^{\prime\prime}>0$ implies that $|U'(x_1-x_{1*})|$ has a maximum at a unique point $x_{1*}$, and is increasing as $|x_1-x_{1*}|$ increases. Without loss of generality, we assume $x_{1*}=0$.

\subsection{Useful inequalities}
In this part, we present two lemmas associated with some weighted Poincar\'e type inequalities, which are used several times in the following sections.
\begin{lemma}\label{lem:poincare}
Let $m(t)$ be any function of $t$, and $\phi_1, \phi_2$ any integrable functions such that $\phi_1\ge 0$, $\int_{\bbr} \phi_2\neq 0$, and $|x_1|\phi_1(x_1)$ and $|x_1|\phi_2(x_1)$ are all integrable on $\mathbb{R}$. If $\phi_2(\cdot+m(t))f \in L^1(\Omega)$ and $\nabla f \in L^2(\Omega)$, then there exists constant $C$ such that
\[
\begin{array}{ll}
\di \int_{\Omega} f^2(x) \phi_1(x_1+m(t)) dx \le C \Big[ \Big(\int_{\Omega} f(x) \phi_2(x_1+m(t)) dx\Big)^2  +  \int_{\Omega} |\nabla f|^2 dx\Big].
\end{array}
\]
\end{lemma}
\begin{proof}
Integrating the following identity w.r.t. $y_1\in \bbr$,
\[
f(x_1,x')\phi_2(y_1+m(t))=f(y_1,x')\phi_2(y_1+m(t))+\int_{y_1}^{x_1} \partial_{x_1}f(z_1,x') dz_1 \phi_2(y_1+m(t)),
\]
yields that
\[
\begin{array}{ll}
\di f(x_1,x')\int_{\bbr} \phi_2(y_1+m(t)) dy_1\\
\di \quad = \int_{\bbr} f(y_1,x')\phi_2(y_1+m(t)) dy_1+\int_{\bbr} \int_{y_1}^{x_1} \partial_{x_1}f(z_1,x') dz_1 \phi_2(y_1+m(t)) dy_1.
\end{array}
\]
Then one has
\[
\begin{array}{ll}
\di f^2(x_1,x')\Big(\int_{\bbr} \phi_2 dy_1\Big)^2\\
\di \le 2\Big(\int_{\bbr} f(y_1,x')\phi_2(y_1+m(t)) dy_1\Big)^2+ 2\Big(\int_{\bbr} \int_{y_1}^{x_1} \partial_{x_1}f(z_1,x') dz_1 \phi_2(y_1+m(t)) dy_1 \Big)^2.
\end{array}
\]
Multiplying the above inequality by $\phi_1(x_1+m(t))$, and then integrating w.r.t. $x:=(x_1,x')\in \Omega$, we have
\begin{align*}
\begin{aligned}
&\di \Big(\int_{\bbr} \phi_2 dy_1\Big)^2\int_{\Omega} f^2(x) \phi_1(x_1+m(t)) dx \\
&\le 2\int_{\bbr}\phi_1(x_1+m(t))dx_1 \int_{\bbt^{N-1}}\Big(\int_{\bbr} f(y_1,x')\phi_2(y_1+m(t)) dy_1\Big)^2dx' \\
&\quad +2\int_{\Omega} \Big(\int_{\bbr} \int_{y_1}^{x_1} \partial_{x_1}f(z_1,x') dz_1 \phi_2(y_1+m(t)) dy_1 \Big)^2 \phi_1(x_1+m(t)) dx\\
&=:I_1+I_2.
\end{aligned}
\end{align*}
Set $H(x',t):=\int_{\bbr} f(y_1,x')\phi_2(y_1+m(t)) dy_1$, and $\bar H(x',t):=H(x',t)-\int_{\bbt^{N-1}} H(z',t)dz'$. Then, Poincar\'e inequality yields
\begin{align*}
\begin{aligned}
\int_{\bbt^{N-1}} |H(x',t)|^2 dx' &\le \int_{\bbt^{N-1}} |\bar H(x',t)|^2 dx' + \Big(\int_{\bbt^{N-1}} H(z',t)dz' \Big)^2\\
&\le C\int_{\bbt^{N-1}} |\partial_{x'}\bar H(x',t)|^2 dx' + \Big(\int_{\bbt^{N-1}} H(z',t)dz' \Big)^2\\
&\le C\int_{\Omega} |\partial_{x'} f|^2 dx + \Big(\int_{\bbt^{N-1}} H(z',t)dz' \Big)^2,
\end{aligned}
\end{align*}
which implies that
\begin{align*}
\begin{aligned}
I_1\le C \int_{\bbr}\phi_1dx_1 \int_{\Omega} |\nabla f|^2 dx+ 2\int_{\bbr}\phi_1dx_1 \Big(\int_{\Omega} f(x) \phi_2(x_1+m(t)) dx\Big)^2
\end{aligned}
\end{align*}
For the estimate on $I_2$, since $\phi_2$ and $|\cdot|\phi_2(\cdot)$ are integrable, we have
\begin{align*}
\begin{aligned}
& \Big(\int_{\bbr} \int_{y_1}^{x_1} \partial_{x_1}f(z_1,x') dz_1 \phi_2(y_1+m(t)) dy_1 \Big)^2\\ 
 &\le  \Big(\int_{\bbr} \|\partial_{x_1}f(\cdot,x')\|_{L^2(\bbr)} |x_1-y_1|^{1/2} \phi_2(y_1+m(t)) dy_1 \Big)^2 \\
 &\le C\|\partial_{x_1}f(\cdot,x')\|_{L^2(\bbr)}^2\int_{\bbr} (|x_1+m(t)|+|y_1+m(t)|) \phi_2(y_1+m(t)) dy_1\\
 &\le C\|\partial_{x_1}f(\cdot,x')\|_{L^2(\bbr)}^2 (|x_1+m(t)|+C),
\end{aligned}
\end{align*}
which together with the integrability of $\phi_1$ and $|\cdot|\phi_1(\cdot)$ implies that
\begin{align*}
\begin{aligned}
I_2\le C\int_{\Omega} \|\partial_{x_1}f(\cdot,x')\|_{L^2(\bbr)}^2 (|x_1+m(t)|+C) \phi_1(x_1+m(t)) dx\le C \int_{\Omega} |\nabla f|^2 dx.
\end{aligned}
\end{align*}
\end{proof}

\begin{lemma}\label{lem:fund}
Let $U$ be a planar shock wave defined by \eqref{shock}, and $m(t)$ any smooth function of $t$, and $\widetilde Y$ any smooth function satisfying $\int_{\Omega} |U^\prime(x_1+m(t))|\widetilde Y(t,x) dx=0$. Then, there exists a constant $C>0$ such that for all $t>0$ and $x\in\Omega$,
\begin{equation*}
\begin{array}{ll}
\di |U^\prime(x_1+m(t))||\widetilde Y(t,x)|^2\\
\di\quad\leq C(|x_1+m(t)|+|U^\prime(x_1+m(t))|)\int_{\bbr} |U^\prime(y_1+m(t))||\partial_{y_1}\widetilde Y(t,y_1,x')|^2dy_1\\
\di \qquad\qquad\qquad\qquad\quad  +C|U^\prime(x_1+m(t))|\int_{\Omega} |U^\prime(y_1+m(t))||\partial_{y'}\widetilde Y(t,y)|^2dy.
\end{array}
\end{equation*}
\end{lemma}
\begin{proof}
Since 
\[
\int_{\Omega} |U^\prime(x_1+m(t))|\widetilde Y(t,x) dx=0,
\]
we have
\begin{equation}
\begin{array}{ll}
\di  |U^\prime(x_1+m(t))||\widetilde Y(t,x)|^2 = |U^\prime(x_1+m(t))| \Big|\widetilde Y(t,x)-\frac{\int_\Omega |U'(y_1+m(t))|\widetilde Ydy}{\int_\Omega |U'(y_1+m(t))|dy} \Big|^2\\
\di \le C|U^\prime(x_1+m(t))|\Big|\int_{\Omega}|U^\prime(y_1+m(t))|\big(\widetilde Y(t,x_1,x^\prime)-\widetilde Y(t,y_1,y^\prime)\big)dy \Big|^2\\
\di  =C|U^\prime(x_1+m(t))| \Big|\int_{\Omega}|U^\prime(y_1+m(t))|\Big(\int_{y_1}^{x_1}\partial_{z_1}\widetilde Y(t,z_1,x^\prime)dz_1+\int^{x^\prime}_{y^\prime}\partial_{z'}\widetilde Y(t,y_1,z')dz'\Big)dy\Big|^2\\[4mm]
\di \leq  C|U^\prime(x_1+m(t))|\Big|\int_{\Omega} |U^\prime(y_1+m(t))| \int_{-m(t)}^{x_1} \partial_{z_1}\widetilde Y(t,z_1,x^\prime)dz_1 dy\Big |^2\\
\di \quad+ C|U^\prime(x_1+m(t))|\Big|\int_{\Omega} |U^\prime(y_1+m(t))| \int^{-m(t)}_{y_1} \partial_{z_1}\widetilde Y(t,z_1,x^\prime)dz_1 dy\Big |^2\\[4mm]
\di \quad + C|U^\prime(x_1+m(t))|\Big|\int_{\Omega} |U^\prime(y_1+m(t))| \int_{y^\prime}^{x^\prime} \partial_{z'}\widetilde Y(t,y_1,z')dz'dy\Big |^2 \\
\di := I_1+I_2+I_3.
\end{array}
\end{equation}
Since $|U'(x_1)|$ is decreasing in $|x_1|$, we have 
\begin{equation}\label{I1}
\begin{array}{ll}
\di I_1\leq C |U^\prime(x_1+m(t))| \Big|\int_{-m(t)}^{x_1} \partial_{z_1}\widetilde Y(t,z_1,x^\prime)dz_1\Big|^2\\[4mm]
\di\quad =C |U^\prime(x_1+m(t))| \Big|\int_0^{x_1+m(t)} \partial_{z_1}\widetilde Y(t,z_1-m(t),x^\prime)dz_1\Big|^2\\
\di \quad \leq C\Big|\int_0^{x_1+m(t)} \sqrt{|U^\prime(z_1)|} |\partial_{z_1} \widetilde Y(t,z_1-m(t),x^\prime)|dz_1\Big|^2\\
\di \quad \leq C|x_1+m(t)| \int_{\mathbb{R}}|U^\prime(z_1+m(t))||\partial_{z_1}\widetilde Y(t,z_1,x^\prime)|^2 dz_1.
\end{array}
\end{equation}
Similarly, we estimate $I_2$ as
\begin{equation}\label{I2}
\begin{array}{ll}
\di I_2\leq  |U^\prime(x_1+m(t))| \Big|\int_\Omega |U^\prime(y_1+m(t))\big|\int_{-m(t)}^{y_1} \partial_{z_1}\widetilde Y(t,z_1,x^\prime)dz_1 dy\Big|^2\\
\di  = |U^\prime(x_1+m(t))| \Big|\int_\Omega |U^\prime(y_1+m(t))|\int_0^{y_1+m(t)} \partial_{z_1}\widetilde Y(t,z_1-m(t),x^\prime)dz_1 dy\Big|^2\\
\di\leq  |U^\prime(x_1+m(t))| \Big|\int_\Omega \sqrt{|U^\prime(y_1+m(t))|}\int_0^{y_1+m(t)} \sqrt{|U^\prime(z_1)|}\partial_{z_1}\widetilde Y(t,z_1-m(t),x^\prime)dz_1 dy\Big|^2\\
\di  \leq |U^\prime(x_1+m(t))| \Big(\int_{\Omega}\sqrt{|y_1+m(t)|}\sqrt{|U^\prime(y_1+m(t))|}dy\Big)^2  \int_{\mathbb{R}}|U^\prime(z_1+m(t))||\partial_{z_1}\widetilde Y(t,z_1,x^\prime)|^2 dz_1\\
\di  \leq C|U^\prime(x_1+m(t))| \int_{\mathbb{R}}|U^\prime(z_1+m(t))||\partial_{z_1}\widetilde Y(t,z_1,x^\prime)|^2 dz_1.
\end{array}
\end{equation}
Since $x'\in\bbt^{N-1}$, using H$\ddot{\mbox{o}}$lder inequality, we have
\begin{equation}\label{I3}
\begin{array}{ll}
\di I_3\leq  |U^\prime(x_1+m(t))| \Big|\int_\Omega |U^\prime(y_1+m(t))|\Big(\int_{\mathbb{T}^{N-1}} |\partial_{z'}\widetilde Y(t,y_1,z')|^2dz'\Big)^{\f12}|y^\prime-x^\prime|^{\f12} dy\Big|^2\\
\di \quad \leq C|U^\prime(x_1+m(t))| \int_\Omega |U^\prime(y_1+m(t))||\partial_{z'}\widetilde Y(t,y_1,z')|^2 dy_1dz'.
\end{array}
\end{equation}
\end{proof}

\section{Proof of Theorem \ref{special}: Special perturbation}\label{sec:special}
\setcounter{equation}{0}
In this section, we prove Theorem \ref{special}. A straightforward computation together with \eqref{Y-eq-0} implies that a special perturbation $u=U(x_1+Y(t,x))$ is a solution of \eqref{eq}, since
\begin{align*}
\begin{aligned} 
&\partial_t u+\mbox{div}A(u)-\Delta u=U'(x_1+Y)\Big( \partial_t Y -A'_1(U(Y+x_1))\partial_{x_1}Y +\sum_{i=2}^{N}A'_i(U( Y+x_1))\partial_{x_i}Y\\
&\quad -A'_1(U( Y+x_1))|\nabla_x Y|^2 -\Delta Y \Big) =0.
\end{aligned}
\end{align*}
We now prove the existence of solutions $Y$ to the equation \eqref{Y-eq-0}. The local existence follows the same arguments as in Appendix. For global-in-time estimates, notice that the new variable $\widetilde Y : =Y-c(t)$, $c(t)$ as in \eqref{ct}, satisfies 
\begin{equation}\label{tY-e}
\left\{
\begin{array}{ll}
\di \partial_t \widetilde Y-A_1^\prime(U(Y+x_1))\partial_{x_1} \widetilde Y+\sum_{i=2}^NA_{i}^\prime(U(Y+x_1))\partial_{x_i} \widetilde Y\\
\di \qquad\qquad\qquad  -A_1^\prime(U(Y+x_1))|\nabla \widetilde Y|^2- \Delta \widetilde Y=-c^\prime(t),\\
\di \widetilde Y(t=0,x)=Y_0(x)-c(0):=\widetilde Y_0(x).
\end{array}\right.
\end{equation}
Multiplying the above equation by $|U'(x_1)| \widetilde Y$, and simple computations yield that
\begin{equation}\label{LL0}
\begin{array}{ll}
\di \partial _t \Big(|U'(x_1)|\f{\widetilde Y^2}{2}\Big)\underbrace{-A_1^\prime (U(x_1)) |U'(x_1)|\partial_{x_1}\big(\f{\widetilde Y^2}{2}\big)}_{J_1}-\Big[A_1'(U( Y+x_1))-A_1'(U(x_1))\Big]|U'(x_1)|\widetilde Y\partial_{x_1}\widetilde Y\\
\di +\underbrace{\sum_{i=2}^N A_i^\prime (U(Y+x_1)) |U'(x_1)|\partial_{x_i}\big(\f{\widetilde Y^2}{2}\big)}_{J_2}
  -A_1^\prime (U( Y+x_1))|\nabla  Y|^2|U'(x_1)|\widetilde Y -{\rm div}(|U'(x_1)| \widetilde Y\nabla\widetilde Y)\\
\di +\partial_{x_1}(\partial_{x_1}|U'(x_1)|\f{\widetilde Y^2}{2}) +|U'(x_1)||\nabla\widetilde Y|^2\underbrace{-\partial^2_{x_1x_1} |U'(x_1)|\f{\widetilde Y^2}{2}}_{J_3}=0.
\end{array}
\end{equation}
Since it follows from \eqref{shock} that the shock profile $U^\prime(x_1)$ satisfies that
\beq\label{U-prime}
|U'(x_1)|''=\Big(A_1^\prime (U(x_1))|U^\prime(x_1)|\Big)',
\eeq
the summation of the two terms $J_1$ and $J_3$ can be computed by
\begin{align*}
\begin{aligned}
J_1+J_3&=-\partial_{x_1}\Big(A_1^\prime (U(x_1)) |U^\prime(x_1)|\f{\widetilde Y^2}{2}\Big).
\end{aligned}
\end{align*}
We rewrite the term $J_2$ as
\[
J_2=\sum_{i=2}^N A_i^\prime (U(\widetilde Y+x_1+c(t))) |U^\prime(x_1)|\partial_{x_i}\Big(\f{(\widetilde Y+x_1+c(t))^2}{2}-(x_1+c(t))(\widetilde Y+x_1+c(t))\Big),
\]
setting $F_i(z):=\int_0^z A_i^\prime (U(s))sds$ and $G_i(z)=\int_0^z A_i^\prime (U(s))ds$ yield that
\[
J_2=\sum_{i=2}^N  |U^\prime(x_1)|\partial_{x_i}\Big(F_i(\widetilde Y+x_1+c(t))-(x_1+c(t))G_i(\widetilde Y+x_1+c(t))\Big),
\]
which vanishes after the integration with respect to $x' \in\mathbb{T}^{N-1}.$
Thus, integrating \eqref{LL0} over $\Omega$ yields that
\begin{equation}\label{LL2-11}
\begin{array}{ll}
\di \frac{d}{dt} \int_\Omega |U^\prime(x_1)|\f{\widetilde Y^2}{2} dx+\int_\Omega |U^\prime(x_1)||\nabla \widetilde Y|^2dx\\
\di\quad =\int_{\Omega} \Big(A_1(U(Y+x_1))-A_1(U(x_1))\Big)|U^\prime(x_1)|\partial_{x_1}\big(\f{\widetilde Y^2}{2}\big)dx+\int_{\Omega}A_1^\prime (U(Y+x_1))|\nabla \widetilde Y|^2|U^\prime(x_1)|\widetilde Y dx\\
\di\quad :=I_1+I_2.
\end{array}
\end{equation}
Notice that thanks to the maximum principle on the equation \eqref{Y-eq-0} as
\[
\|Y\|_{L^\infty((0,\infty)\times \Omega)}\leq \|Y_0\|_{L^\infty(\Omega)},
\]
it holds that for any $t\geq 0,$
$$
|c(t)|\leq \|Y\|_{L^\infty((0,\infty)\times \Omega)}\leq \|Y_0\|_{L^\infty(\Omega)},
$$
which yields that
$$
\|\widetilde{Y}\|_{L^\infty((0,\infty)\times \Omega)} \leq \|Y\|_{L^\infty((0,\infty)\times \Omega)}+|c(t)|_{L^\infty(0,\infty)}\leq 
2\|Y_0\|_{L^\infty(\Omega)}.
$$
Therefore, $I_2$ is estimated as
\begin{align*}
\begin{aligned}
 |I_2|&\leq C\|\widetilde Y\|_{L^\infty}\int_\Omega |U^\prime(x_1)| |\nabla \tilde Y|^2 dx\\
&\leq 2C\|Y_0\|_{L^\infty} \int_\Omega |U^\prime(x_1)| |\nabla \tilde Y|^2 dx.
\end{aligned}
\end{align*}
For the first term $I_1$, we use Lemma \ref{lem:fund} with $m(t)\equiv 0$ to estimate
\begin{align*}
\begin{aligned}
|I_1|&\leq  \int_\Omega \int_0^1|U^\prime(x_1+\theta Y)|d\theta |Y| |U^\prime(x_1)||\partial_{x_1}\widetilde Y||\widetilde Y|dx\\
& \leq C\|Y\|_{L^\infty}\Big[\int_\Omega  |U^\prime(x_1)||\partial_{x_1}\widetilde Y|^2dx+\int_\Omega\int_0^1|U^\prime(x_1+\theta\widetilde Y)|^2 |U^\prime(x_1)||\widetilde Y|^2 d\theta dx\Big]\\
& \leq C\|Y_0\|_{L^\infty}\int_\Omega  |U^\prime(x_1)||\partial_{x_1}\widetilde Y|^2dx\\
&\quad + C\|Y_0\|_{L^\infty}\int_\Omega\int_0^1  |U^\prime(x_1+\theta\widetilde Y)|^2\Big[(|x_1|+|U^\prime(x_1)|)\int_{\bbr} |U^\prime(y_1)||\partial_{y_1}\widetilde Y(t,y_1,x')|^2dy_1\\
&\qquad +|U^\prime(x_1)|\int_{\Omega} |U^\prime(y_1)||\partial_{y'}\widetilde Y(t,y)|^2dy\Big]d\theta dx\\
&\leq  C\|Y_0\|_{L^\infty}\int_\Omega  |U^\prime(x_1)||\nabla \widetilde Y|^2dx,
\end{aligned}
\end{align*}
Taking $\|Y_0\|_{L^\infty}\ll1$ yields that
\beq\label{Y-ex}
 \frac{d}{dt} \int_{\Omega} |U^\prime(x_1)|\f{\widetilde Y^2}{2} dx+ \int_{\Omega} |U^\prime(x_1)||\nabla \widetilde Y|^2dx\le0.
\eeq
Since
\[
\int_{\Omega} |U^\prime(x_1)|\f{\widetilde Y_0^2}{2} dx \le 2\|\widetilde Y_0\|_{L^{\infty}(\Omega)}^2\int_{\Omega} |U^\prime(x_1)| dx
\]
we completes \eqref{Y-reg}.\\

With the weighted estimates \eqref{Y-reg}, we can first show the large-time behavior of the shift $\widetilde Y$ and then prove the $L^2$ stability of  viscous shock profile for the special perturbation.
Set
\[
F(t):=\int|U^\prime(x_1)|^2|\widetilde Y(t,x)|^2 dx.
\]
We want to show that
\begin{equation}\label{gl}
\lim_{t\rightarrow +\infty} F(t)=0.
\end{equation}
Using Lemma \ref{lem:fund} with $m(t)\equiv 0$, and then using \eqref{Y-ex}, we have
\beq\label{F-1}
 \int_0^{\infty} F(t)dt \leq C\int_0^{\infty}\int_\Omega |U^\prime(x_1)||\nabla\widetilde Y(t,x)|^2 dxdt \le C.
\eeq
On the other hand, it follows from \eqref{Y-ex} that $F(t)$ is decreasing in time $t$, and therefore,
\[
\int_0^{t} |F'(s)| ds \le F(0)-F(t)\le F(0), ~t>0,
\]
which implies that $F^\prime\in L^1(0,+\infty).$\\
Therefore, \eqref{gl} holds true.
Then we have
\begin{align*}
\begin{aligned}
\int_{\Omega} |U(x_1+Y(t,x))-U(x_1+c(t)) |^2 dx& \leq C\int_\Omega \int_0^1 |U^\prime(x_1+\theta Y+(1-\theta)c(t))|^2 
|\widetilde Y|^2d\theta dx\\
& \leq C\int |U^\prime(x_1)|^2 |\widetilde Y|^2dx \rightarrow 0, ~~{\rm as}~~ t\rightarrow+\infty,
\end{aligned}
\end{align*}
which completed the proof of Theorem \ref{special}.

\section{Proof of Theorem \ref{thm:general} and Remark \ref{remark-m} : General perturbation}
\setcounter{equation}{0}
In this section, we present proofs of Theorem \ref{thm:general} and the claim in Remark \ref{remark-m}. Since 
the initial assumption (on smallness of $\|u_0-U\|_{H^s(\Omega)}$) in Remark \ref{remark-m} is stronger than the one in Theorem \ref{thm:general}, we first prove the claim in Remark \ref{remark-m} and then Theorem \ref{thm:general}.

As stated in Theorem \ref{thm:general} and Remark \ref{remark-m}, we aim to show that the perturbation 
\[
u(t,x)-U(x_1+ Y(t,x))
\]
is non-increasing in time.\\
For that, we first derive an equation on $V(t,x):=U(x_1+ Y(t,x))$. Using \eqref{shock}, \eqref{Y-eq} and the chain rule, we find that $V$ satisfies the equation \eqref{VG} with
$$
G = U'( Y+x_1)\Big(w_{1}(1-\psi_M (x_1+m(t))) - h_M(t)(1-\psi_M (x_1+m(t))) -g(t) \Big),\\
$$
and the initial value $V(0,x)=U(x_1).$ That is, 
\begin{align}
\begin{aligned} \label{V-eq}
&\partial_t V + {\rm div} A(V) + w\cdot \nabla V -\Delta V\\
&\quad = U'( Y+x_1)\Big(w_{1}(1-\psi_M (x_1+m(t))) - h_M(t)(1-\psi_M (x_1+m(t))) -g(t) \Big),\\
&V(0,x)=U(x_1).
\end{aligned}
\end{align}
Therefore, it follows from \eqref{ineq-0} that
\begin{align}
\begin{aligned} \label{ge-uV-1}
&\frac{1}{2}\frac{d}{dt}\int_{\Omega} |u-V|^2 dx + \int_{\Omega} |\nabla(u-V)|^2 dx \\
&\quad =  -\int_{\Omega}\Big(A(u|V) - (u-V) w\Big)\cdot \nabla V dx\\
&\quad+ \int_{\Omega} (u-V)U'(Y+x_1)\Big(w_{1}(1-\psi_M (x_1+m(t))) - h_M(t)(1-\psi_M (x_1+m(t))) -g(t)  \Big)dx.
\end{aligned}
\end{align}

\subsection{A priori estimate on $u-V$}
In this part, we show a $L^2$-contraction of $u-V$ under an a priori assumption that $\nabla Y$ is uniformly small in $(t,x)\in (0,T)\times\Omega$ for any fixed $T>0$. Then, in the next steps, we shall prove a global-in-time existence of $Y$ in suitable spaces, for which the a prior assumption on $\nabla Y$ is guaranteed.\\
We first get a $L^2$-contraction of $u-V$ in the case of $\varphi\equiv1$ in \eqref{Y-eq} (for Remark \ref{remark-m}). In the sequel, $T$ denotes any positive constant.
\begin{lemma}\label{lem:cont}
Let $Y$ be a solution of \eqref{Y-eq} with $\varphi\equiv 1$ for all $t>0$. Assume there exists $\eps_0>0$ small enough such that
\beq\label{DY-infty}
\|\nabla Y\|_{L^{\infty}((0,T)\times\Omega)}<\eps_0.
\eeq
Then, for all $t\in [0,T]$,
\[
\frac{1}{2}\int_\Omega (u-V)^2dx +\int_0^T\int_\Omega |\nabla(u-V)|^2 dx dt+\int_0^T\Big(\int_\Omega(u-V)U^\prime(x_1+m(t)) dx\Big)^2 dt  \le\frac{1}{2}\int_\Omega (u_0-U)^2dx.
\]
\end{lemma}
\begin{proof}
First of all, since $w=\frac{A(u|V)}{u-V}$, it follows from \eqref{ge-uV-1} that
\begin{align*}
\begin{aligned} 
&\frac{1}{2}\frac{d}{dt}\int_{\Omega} |u-V|^2 dx + \int_{\Omega} |\nabla(u-V)|^2 dx \\
&\quad =  \int_{\Omega} (u-V)U'(Y+x_1)\Big(w_{1}(1-\psi_M (x_1+m(t))) - h_M(t)(1-\psi_M (x_1+m(t))) -g(t)  \Big)dx.
\end{aligned}
\end{align*}
Then, we derive the other dissipation term $\Big(\int_\Omega(u-V)U^\prime(x_1+m(t)) dx\Big)^2$ from the above last term related to $g(t)$ as follows:
\begin{equation}\label{u2}
\begin{array}{ll}
\di  \frac 12\frac{d}{dt}\int_\Omega |u-V|^2dx+\int_\Omega |\nabla(u-V)|^2 dx+\Big(\int_\Omega(u-V)U^\prime(x_1+m(t)) dx\Big)^2\\[3mm]
\di =-\int_\Omega (u-V)U^\prime(x_1+m(t))dx \int_\Omega (u-V)\big(U^\prime(Y+x_1)-U^\prime(x_1+m(t))\big)dx\\[3mm]
\di \quad +\int_\Omega (u-V) U^\prime( Y+x_1)w_1(1-\psi_M (x_1+m(t)))dx\\[3mm]
\di \quad +\frac{1}{2(M+1)}\int_{|x_1+m(t)|\leq M+1} w_1 dx \int_\Omega(u-V) U^\prime( Y+x_1)(1-\psi_M (x_1+m(t))) dx\\
\di :=J_1+J_2+J_3.
\end{array}
\end{equation}
In the sequel, we often use the notation $\widetilde Y$ to denote $\widetilde Y:= Y-m(t)$.\\
We first estimate $J_1$ as
\begin{align*}
\begin{aligned}
|J_1| &\leq \frac{1}{2} \Big(\int_\Omega(u-V)U^\prime(x_1+m(t)) dx\Big)^2+ \frac{1}{2} \Big(\int_\Omega (u-V)\big(U^\prime(Y+x_1)-U^\prime(x_1+m(t))\big)dx\Big)^2\\[3mm]
& \le  \frac{1}{2} \Big(\int_\Omega(u-V)U^\prime(x_1+m(t)) dx\Big)^2+ \frac{1}{2} \underbrace{\Big(\int_\Omega |u-V|\int_0^1|U^{\prime\prime}(\theta  \widetilde Y+x_1+m(t))|d\theta |\widetilde Y|dx\Big)^2}_{L}. 
\end{aligned}
\end{align*}
To control the second term $L$ above, we use the following estimates
\begin{align}
\begin{aligned}\label{Y-est-0}
|\widetilde Y(t,x)|&\le \Big|Y(t,x)-\frac{\int_\Omega |U'(y_1+m(t))|Ydy}{\int_\Omega |U'(y_1+m(t))|dy} \Big|\\
&\le C \int_{\Omega} |U^\prime(y_1+m(t))||Y(t,x)- Y(t,y)| dy \\
&\le C\|\nabla Y\|_{L^{\infty}} \int_{\Omega} |U^\prime(y_1+m(t))|(|x_1+m(t)|+|y_1+m(t)|+C) dy \\
&\le C\eps_0 (|x_1+m(t)|+1),
\end{aligned}
\end{align}
where we have used the assumption $\|\nabla Y\|_{L^{\infty}((0,T)\times\Omega)}<\eps_0$.\\
Taking $\eps_0$ sufficiently small such that $C\eps_0<\frac{1}{3}$, we have that for all $\theta\in[0,1]$,
\[
|\theta \widetilde Y+x_1+m(t)|\ge |x_1+m(t)|-| \widetilde Y|\ge \frac{2|x_1+m(t)|}{3}-C,
\]
which together with \eqref{shock} and \eqref{U-prop} implies that
\beq\label{ine-1}
|U^{\prime\prime}(\theta \widetilde Y+x_1+m(t))|^{\f32}\le C|U'(\theta \widetilde Y+x_1+m(t))|^{\f32}\le  C|U'(x_1+m(t))|.
\eeq
Therefore, we have
\begin{align*}
\begin{aligned}
L&\le \eps_0^2\int_{\Omega} |u-V|^2|U^{\prime\prime}(\theta \widetilde Y+x_1+m(t))|^{\f32}dx \int_{\Omega} |U^{\prime\prime}(\theta \tilde Y+x_1+m(t))|^{\f12}(|x_1+m(t)|+C)^2dx\\
&\le C\eps_0^2\int_{\Omega} |u-V|^2|U'(x_1+m(t))|dx.
\end{aligned}
\end{align*}
We now use Lemma \ref{lem:poincare} with taking $\phi_1=|U'|$ and $\phi_2=U'$, to get
\begin{align*}
\begin{aligned}
L & \le  C\eps_0^2\Big(\int_\Omega(u-V)U^\prime(x_1+m(t)) dx\Big)^2+C\eps_0^2\|\nabla(u-V)\|_{L^2(\Omega)}^2. 
\end{aligned}
\end{align*}
For the term $J_2$, since
\[
J_2=\int_\Omega (u-V) U^\prime( \widetilde Y+m(t)+x_1)w_1(1-\psi_M (x_1+m(t)))dx
\]
we use the same estimates as the term $L$ to get
\begin{align*}
\begin{aligned}
|J_2| & \le  C\int_\Omega |u-V|^2 |U^\prime(x_1+m(t))|^{2/3}(1-\psi_M (x_1+m(t)))dx, 
\end{aligned}
\end{align*}
where we have used $|w|\le C|u-V|$. Then, using Lemma \ref{lem:poincare} with taking $\phi_1=|U'|^{2/3}(1-\psi_M)$ and $\phi_2=U'$, and taking $M$ to be suffciently large, we have
\begin{align*}
\begin{aligned}
|J_2| & \le  \frac{1}{4}\Big(\int_\Omega(u-V)U^\prime(x_1+m(t)) dx\Big)^2+\frac{1}{4}\|\nabla(u-V)\|_{L^2((0,T)\times\Omega)}^2. 
\end{aligned}
\end{align*}
Likewise, since
\[
|J_3|\leq\frac{C}{2(M+1)} \underbrace{\int_{|x_1+m(t)|\leq M+1} |u-V| dx}_{J_{31}} \underbrace{\int_\Omega |u-V| |U^\prime(Y+x_1)|(1-\psi_M (x_1+m(t)))dx,}_{J_{32}}
\]
Holder inequality and \eqref{U-prop} yield that
\begin{align}
\begin{aligned}\label{J31}
|J_{31}|&\leq \sqrt{2(M+1)}\Big(\int_{|x_1+m(t)|\leq M+1} |u-V|^2dx\Big)^{\f12}\\
&\leq C\sqrt{M+1} e^{\frac{c_{\pm}}{2}(M+1)} \Big(\int_{|x_1+m(t)|\leq M+1} |u-V|^2|U^\prime(x_1+m(t))| dx\Big)^{\f12}\\
&\leq C\sqrt{M+1} e^{\frac{c_{\pm}}{2}(M+1)} \Big(\int_\Omega |u-V|^2|U^\prime(x_1+m(t))| dx\Big)^{\f12},
\end{aligned}
\end{align}
and \eqref{ine-1} yields that
\begin{align}
\begin{aligned}\label{J32}
|J_{32}|&\le \Big( \int_\Omega|u-V|^2 
|U^\prime(\widetilde Y+x_1+m(t))|^{\f12}(1-\psi_M (x_1+m(t)))dx\Big)^{\f12}\\
&\quad \times \Big( \int_\Omega  |U^\prime(\widetilde Y+x_1+m(t))|^{\f32}(1-\psi_M (x_1+m(t)))dx\Big)^{\f12}\\
&\le\Big( \int_\Omega|u-V|^2 |U^\prime(x_1+m(t))|^{\f13}(1-\psi_M (x_1+m(t)))dx\Big)^{\f12}\\
&\quad\times\Big(\int_\Omega  |U^\prime(x_1+m(t))|(1-\psi_M (x_1+m(t)))dx\Big)^{\f12}\\
&\le C e^{-\frac{c_{\pm}}{2}M}\Big( \int_\Omega|u-V|^2 |U^\prime(x_1+m(t))|^{\f13}(1-\psi_M (x_1+m(t)))dx\Big)^{\f12}.
\end{aligned}
\end{align}
Then we apply Lemma \ref{lem:poincare} with $\phi_1=|U^\prime|$, $\phi_2=U^\prime$ to \eqref{J31}, and $\phi_1=|U^\prime|^{1/3}(1-\psi_M)$, $\phi_2=U^\prime$ to \eqref{J32} so that
\begin{align*}
\begin{aligned}
|J_3|\leq\frac{C}{\sqrt{M+1}}\Big(\int_\Omega(u-V)U^\prime(x_1+m(t)) dx\Big)^2+ \frac{C}{\sqrt{M+1}}\|\nabla(u-V)\|_{L^2(\Omega)}^2.
\end{aligned}
\end{align*}
Therefore, combining all estimates above together with taking small $\eps_0$ and large $M$, we have
\[
\frac{d}{dt}\int_\Omega (u-V)^2dx+\int_\Omega |\nabla(u-V)|^2 dx+\Big(\int_\Omega(u-V)U^\prime(x_1) dx\Big)^2\le 0,
\]
which completes the proof.
\end{proof}

The following Lemma provides a $L^2$-contraction of $u-V$ when the shift $Y$ is a solution of \eqref{Y-eq}.
\begin{lemma}\label{lem:cont-2}
For any fixed $t_0\in (0,T)$, let $Y$ be a solution of \eqref{Y-eq}. Assume there exists $\eps_0>0$ small enough such that
\beq\label{Y-aa}
\|\nabla Y\|_{L^{\infty}((0,T)\times\Omega)}<\eps_0.
\eeq
Then, for all $t\le t_0$, there exists a constant $C_0$ depending on $t_0$ such that 
\[
\int_{\Omega} |u(t,x)-V(t,x)|^2 dx \le C_0\int_{\Omega} |u_0(x)-U(x_1)|^2 dx,
\]
and for all $t\ge t_0$,
\beq\label{cont-t0}
\frac{d}{dt}\int_\Omega (u-V)^2dx+\int_\Omega |\nabla(u-V)|^2 dx+\Big(\int_\Omega(u-V)U^\prime(x_1) dx\Big)^2\le 0,
\eeq
\end{lemma}
\begin{proof}
First of all, since $\varphi(t)=1$ for all $t\ge t_0$, we have the same estimates as in Lemma \ref{lem:cont}, and thus complete \eqref{cont-t0}. On the other hand, since $\varphi(t)<1$ for all $t< t_0$, we start with \eqref{ineq-0}: 
\begin{equation}\label{u2}
\begin{array}{ll}
\di \frac{d}{dt}\int_\Omega \frac 12 (u-V)^2dx+\int_\Omega |\nabla(u-V)|^2 dx\\[3mm]
\di =-\int_{\Omega}\Big(A(u|V) - (u-V) w\Big)\cdot \nabla V dx+\int_\Omega (u-V) U^\prime( Y+x_1)w_1(1-\psi_M (x_1+m(t)))dx\\[3mm]
\di \quad - \int_\Omega(u-V) U^\prime( Y+x_1)\Big(h_M(t)(1-\psi_M (x_1+m(t))) +g(t)\Big) dx
\di :=I_1+I_2+I_3.
\end{array}
\end{equation}
Since $A(u|V)\le C|u-V|^2$, and thus $|w|\le C|u-V|$, the first term $I_1$ can be estimated as
\begin{align*}
\begin{aligned}
|I_1| \leq \int_{\Omega} |u-V|^2 |U^\prime( Y+x_1)|(|\nabla Y|+1) dx \le C \int_{\Omega} |u-V|^2 dx,
\end{aligned}
\end{align*}
and the second term $I_2$ can be estimated as
\begin{align*}
\begin{aligned}
|I_2| \leq C\int_{\Omega} |u-V|^2 |U^\prime( Y+x_1)| dx \le C \int_{\Omega} |u-V|^2 dx. 
\end{aligned}
\end{align*}
Since $|h_M|\le \frac{C}{\sqrt{M+1}} \|u-V\|_{L^2(\Omega)}$ and $|g|\le C\|u-V\|_{L^2(\Omega)}$, moreover \eqref{ine-1} yields
\[
\Big|\int_\Omega(u-V) U^\prime( Y+x_1) dx\Big| \le  C\|u-V\|_{L^2(\Omega)},
\]
we have
\begin{align*}
\begin{aligned}
|I_3| \leq  C \int_{\Omega} |u-V|^2 dx. 
\end{aligned}
\end{align*}
Therefore, we can use the Gronwall inequality for $t\le t_0$, which completes the proof.
\end{proof}

\subsection{Local existence and a prior estimates on $Y$}
In order to complete a global-in-time $L^2$-contraction from Lemma \ref{lem:cont}, we should estimate the assumptions \eqref{DY-infty} and \eqref{Y-aa} on $\nabla Y$. Therefore, we will prove a global-in-time existence on the shift $Y$ in suitable spaces, for which $\nabla Y$ is uniformly small in $(t,x)\in (0,\infty)\times\Omega$. For that, we first present a local-in-time existence as follows. We present its proof in Appendix.
\begin{proposition}(Local existence)\label{prop:local} 
If $u_0\in L^{\infty}(\Omega)$, then for any $R>0$, there exists $T_0\in (0,\frac{t_0}{2}]$ such that \eqref{Y-eq} has a solution $Y$ satisfying
\beq\label{Y-local-sol}
\|\sqrt{|U'(\cdot + m(t))|} Y\|_{L^{\infty}(0,T_0; L^2(\Omega))}+ \|\nabla Y\|_{L^{\infty}(0,T_0; H^s(\Omega))}+ \|\Delta Y\|_{L^2(0,T_0; H^{s}(\Omega))}\le R,
\eeq
where $s>\frac{N}{2}$.\\
In particular, if $\nabla u_0\in H^{s-1}(\Omega)$ and $u_0\in L^{\infty}(\Omega)$, there exists $T_0>0$ such that \eqref{Y-eq} with $\varphi\equiv1$ has a solution $Y$ satisfying \eqref{Y-local-sol}.
\end{proposition}

In order to prove the global existence on the shift $Y$, we use the continuation argument. For that, we present the following a priori estimates.

\begin{proposition}(A priori estimates)\label{prop:priori} 
Let $Y$ be a solution of \eqref{Y-eq} with $\varphi\equiv1$ for all $t>0$. Assume that there exists $\eps_0>0$ small enough such that
\begin{subequations}
\begin{align}
&\|\nabla Y\|_{L^{\infty}(0,T; L^2(\Omega))}+ \|\Delta Y\|_{L^2((0,T)\times\Omega)}\le\eps_0, \label{ass-2}\\
& \|\nabla Y\|_{L^{\infty}(0,T; H^{s}_{loc}(\Omega))}\le\eps_0, \label{ass-3}\\
&\|u_0-U\|_{H^s(\Omega)}\le\eps_0^{3/2}\label{ass-4},\quad s>\frac{N}{2}.
\end{align}
\end{subequations}
Then, there exists $C>0$ depending only on $s, N$ such that 
\begin{subequations}
\begin{align}
&\|\sqrt{|U'(\cdot+m(t))|} (Y-m(t))\|_{L^{\infty}(0,T; L^2(\Omega))}+\|\nabla Y\|_{L^{\infty}(0,T; L^2(\Omega))}+ \|\Delta Y\|_{L^2((0,T)\times\Omega)} \le C\eps_0^{3/2},\label{lower-Y-1} \\
&\|\nabla  Y\|_{L^{\infty}(0,\infty; H^s_{loc}(\Omega))}+ \|\Delta  Y\|_{L^2(0,\infty; H^{s}_{loc}(\Omega))}\le C\eps_0^{3/2}. \label{higher-Y-1}
\end{align}
\end{subequations}
\end{proposition}

\begin{proposition}(A priori estimates)\label{prop:priori2} 
For any fixed $t_0>0$, let $Y$ be a solution of \eqref{Y-eq}. Assume that there exists $\eps_0>0$ small enough such that \eqref{ass-2} and \eqref{ass-3} with $s>\frac{N}{2}$, and 
\beq\label{ass-5}
\|u_0-U\|_{L^2(\Omega)}\le\eps_0^{3/2}, \quad u_0\in L^{\infty}(\Omega).
\eeq
Then, there exists $C>0$ depending only on $s, N$ and $t_0$ such that \eqref{lower-Y-1} and \eqref{higher-Y-1}.
\end{proposition} 

The next subsections are devoted to the proofs of Proposition \ref{prop:priori} and Proposition \ref{prop:priori2}.

\subsection{Proof of \eqref{lower-Y-1} in Proposition \ref{prop:priori} and \ref{prop:priori2} }

We first obtain a weighted $L^2$ estimates for $Y$ in the first term of the estimate \eqref{lower-Y-1}. For that, we use the assumptions \eqref{ass-2}, \eqref{ass-3} and \eqref{ass-5}, but do not need the smallness of the higher regularity $\nabla(u_0-U)\in H^{s-1}(\Omega)$. Notice that \eqref{ass-4} implies $u_0\in L^{\infty}(\Omega)$.

\begin{lemma}\label{lem:Y}
Let $Y$ be a solution of either \eqref{Y-eq} or \eqref{Y-eq} with $\varphi= 1$ for all $t>0$. Assume \eqref{ass-2}, \eqref{ass-3} and \eqref{ass-5}. Then, there exists a constant $C>0$ such that
\begin{align}
\begin{aligned} \label{Y-est}
\int_{\Omega}  |U^\prime(x_1+m(t))| (Y-m(t))^2 dx + \int_0^t \int_{\Omega} |U^\prime(x_1+m(t))|  |\nabla Y|^2 dx ds
\le C\eps_0^3,\quad \forall t\in (0, T].
\end{aligned}
\end{align}
\end{lemma}
\begin{proof}
For notational simplification, we set $\tilde Y:=Y-m(t)$, and 
then rewrite the equation \eqref{Y-eq} into the form:
\begin{align*}
\begin{aligned}
&\partial_t \widetilde Y -A'_1(U( Y+x_1))\partial_{x_1} \widetilde Y +\sum_{i=2}^{N}A'_i(U( Y+x_1))\partial_{x_i}\widetilde Y\\
&\quad -A'_1(U(Y+x_1))|\nabla_x  Y|^2+w\cdot \nabla_x  Y-\Delta \widetilde  Y \\
& \quad  = -(w_{1}- h_M(t)) \psi_M (x_1+m(t)) - h_M(t) -g(t)-m'(t).
\end{aligned}
\end{align*}
Multiplying the above equation by $|U'(x_1+m(t))| \widetilde Y$, and using the same computations as in Section \ref{sec:special}, we have that
\begin{equation}\label{q-e1}
\begin{array}{ll}
\di \frac{d}{dt}\int_\Omega \frac12 |U^\prime(x_1+m(t))|\widetilde Y^2dx+\int_\Omega |U^\prime(x_1+m(t))||\nabla  Y|^2 dx=-\int_\Omega U^{\prime\prime}(x_1+m(t))m^\prime (t)\f{\widetilde Y^2}{2}dx\\
\di\qquad +\int_\Omega(A_1^\prime(U(Y+x_1)) -A_1^\prime(U(x_1+m(t))))|U^\prime(x_1+m(t))|\widetilde Y\partial_{x_1} \widetilde Ydx\\[3mm]
\di \qquad +\int_\Omega A_1^\prime(U( Y+x_1))|U^\prime(x_1+m(t))|\widetilde Y|\nabla Y|^2 dx -\int_\Omega\omega\cdot \nabla Y |U^\prime(x_1+m(t))|\widetilde Y dx\\
\di\qquad -\int_\Omega  (w_{1}- h_M(t)) \psi_M (x_1+m(t)) |U^\prime(x_1+m(t))|\widetilde Y dx\\
\di \qquad -(h_M(t)+g(t)+m'(t)) \int_\Omega  |U^\prime(x_1+m(t))|\widetilde Y dx:=\sum_{i=1}^6 I_i.
\end{array}
\end{equation}
Since the assumption \eqref{ass-3} implies that $\|\nabla Y\|_{L^{\infty}((0,T)\times\Omega)}\leq C\eps_0\ll1$, it follows from \eqref{mt2} that
\begin{equation}\label{m-est}
\begin{array}{ll}
\di |m^\prime(t)|\leq C\Big[ \int_\Omega |U^\prime(x_1+m(t))| \Big(|A'_1(U(  Y+x_1))||\partial_{x_1} Y|+  \sum_{i=1}^{N}|A_i^\prime(U(Y+x_1))||\partial_{x_i} Y|\Big)dx\\
\di \quad +\int_\Omega |U^\prime(x_1+m(t))| |A_1^\prime(U(Y+x_1))||\nabla Y|^2dx +\int_\Omega |U^\prime(x_1+m(t))| |\omega||\nabla Y|dx\\
\di \quad +\Big|\int_\Omega |U^\prime(x_1+m(t))| \Delta Y dx\Big|+\int_\Omega|U^\prime(x_1+m(t))| |\omega_1-h_M(t)|\psi_M(x_1+m(t)) dx\\
\di\quad +|h_M(t)|+|g(t)|\Big]:=\sum_{i=1}^7 K_i.
\end{array}
\end{equation}
First, by Holder inequality, one has
$$
K_1\leq C \|\sqrt{|U^\prime(x_1+m(t))|}\nabla Y\|_{L^2(\Omega)},
$$
$$
K_2\leq C \|\sqrt{|U^\prime(x_1+m(t))|}\nabla Y\|^2_{L^2(\Omega)},
$$
and
$$
K_3\leq C \|\sqrt{U^\prime(x_1+m(t))}\nabla Y\|_{L^2(\Omega)}\|u-V\|_{L^2(\Omega)}.
$$
For $K_4$, integration by parts and Holder inequality give that
$$
K_4=|\int_\Omega U^{\prime\prime}(x_1+m(t)) \partial_{x_1} Y dx|\leq C  \|\sqrt{|U^\prime(x_1+m(t))|}\nabla Y\|_{L^2(\Omega)}.
$$
We use the same argument as in \eqref{J31} with \eqref{U-prop} to estimate $K_6$ as 
\begin{equation}\label{K6}
\begin{array}{ll}
\di K_6\leq C_M\int_\Omega|u-V||U^\prime(x_1+m(t))|^2 dx\leq C_M\| |U^{\prime}(x_1+m(t))| (u-V)\|_{L^2(\Omega)}\\
\di \quad ~ \leq C_M \Big|\int_\Omega U^{\prime}(x_1+m(t)) (u-V)dx\Big|+C_M\|\nabla(u-V)\|_{L^2(\Omega)},
\end{array}
\end{equation}
where we have used Lemma \ref{lem:poincare} with $\phi_1=|U^\prime|$ and $\phi_2=U^\prime$. \\
Likewise, we have 
$$
\begin{array}{ll}
\di K_5\leq C\| |U^{\prime}(x_1+m(t))| (u-V)\|_{L^2(\Omega)}+C|h_M(t)|\\
\di \quad~ \leq C_M \Big|\int_\Omega U^{\prime}(x_1+m(t)) (u-V)dx\Big|+C_M\|\nabla(u-V)\|_{L^2(\Omega)}.
\end{array}
$$
Therefore, we use the assumption \eqref{ass-2} and Lemma \ref{lem:cont} to get
\begin{equation}\label{mp}
\begin{array}{ll}
\di |m^\prime(t)|\leq C\|\sqrt{|U^\prime(x_1+m(t))|} \nabla Y\|_{L^2(\Omega)}\Big(1+\|\sqrt{|U^\prime(x_1+m(t))|} \nabla Y\|_{L^2(\Omega)}+\|u-V\|_{L^2(\Omega)}\Big)\\
\di \qquad\qquad +C\Big|\int_\Omega U^{\prime}(x_1+m(t)) (u-V)dx\Big|+C\|\nabla (u-V)\|_{L^2(\Omega)}\\
\di \qquad\quad\le C\|\sqrt{|U^\prime(x_1+m(t))|} \nabla Y\|_{L^2(\Omega)}\Big(1+\|\nabla Y\|_{L^{\infty}(0,T;L^2(\Omega))}+\|u-V\|_{L^{\infty}(0,T;L^2(\Omega))}\Big)\\
\di \qquad\qquad +C\Big|\int_\Omega U^{\prime}(x_1+m(t)) (u-V)dx\Big|+C\|\nabla (u-V)\|_{L^2(\Omega)}\\
\di \qquad\quad\le C\|\sqrt{|U^\prime(x_1+m(t))|} \nabla Y\|_{L^2(\Omega)}+C\Big|\int_\Omega U^{\prime}(x_1+m(t)) (u-V)dx\Big|+C\|\nabla (u-V)\|_{L^2(\Omega)}.
\end{array}
\end{equation}
Then, by using the fact \eqref{Y-est-0}  and Lemma \ref{lem:fund}, we can estimate $I_1$ as
\begin{align*}
\begin{aligned}
|I_{1}|&\leq C |m^\prime(t)|\int_\Omega|U^\prime(x_1+m(t))|\widetilde Y^2 dx\\
&\leq C|m^\prime(t)|\Big(\int_\Omega|U^\prime(x_1+m(t))|^{\f32}\widetilde Y^4 dx\Big)^{\f12}\\
&\leq C\eps_0|m^\prime(t)|\Big(\int_\Omega|U^\prime(x_1+m(t))|^{\f32}\widetilde Y^2 (|x_1+m(t)|^2+1) dx\Big)^{\f12}\\
&\leq C\eps_0|m^\prime(t)|\|\sqrt{|U^\prime(x_1+m(t))|}\nabla Y\|_{L^2(\Omega)}\\
&\leq C\eps_0\|\sqrt{|U^\prime(x_1+m(t))|}\nabla Y\|_{L^2(\Omega)}^2+C\Big(\int_\Omega U^{\prime}(x_1+m(t)) (u-V)dx\Big)^2+C\|\nabla (u-V)\|_{L^2(\Omega)}^2.
\end{aligned}
\end{align*}
For $I_2$, \eqref{ine-1} and Lemma \ref{lem:fund} yield that
\begin{align*}
\begin{aligned}
|I_2|&=\Big|\int_\Omega \int_0^1A_1''(U(\theta \widetilde Y+x_1+m(t)))U'(\theta \widetilde Y+x_1+m(t))d\theta ~\widetilde Y |U^\prime(x_1+m(t))|\widetilde Y\partial_{x_1} \widetilde Ydx \Big|\\
&\le C\int_\Omega |U'(x_1+m(t))|^{5/3}|\widetilde Y|^2 |\partial_{x_1} \widetilde Y| dx\\
&\leq C\|\partial_{x_1} \widetilde Y\|_{L^2(\Omega)}\||U^\prime(x_1+m(t))|~ |\widetilde Y|^2|U^\prime(x_1+m(t))|^{2/3}\|_{L^2(\Omega)}\\[3mm]
&\leq C\sup_{t\in[0,T]}\|\nabla Y\|_{L^2(\Omega)} \|\sqrt{|U^\prime(x_1+m(t))|}~\nabla Y\|^2_{L^2(\Omega)}\le C\eps_0 \|\sqrt{|U^\prime(x_1+m(t))|}~\nabla Y\|^2_{L^2(\Omega)},
\end{aligned}
\end{align*}
where we have used \eqref{ass-2} in the last inequality.\\
We use Lemma \ref{lem:fund} to estimate
\begin{equation*}
\begin{array}{ll}
\di |I_3|\leq C\|U^\prime(x_1+m(t))\widetilde Y\|_{L^2(\Omega)}\||\nabla Y|^{1+\f2N}\|_{L^2(\Omega)}\||\nabla Y|^{1-\f2N}\|_{L^\infty(\Omega)}\\[3mm]
\di \qquad \leq C \|\sqrt{|U^\prime(x_1+m(t))|}\nabla \widetilde Y\|_{L^2(\Omega)}
\|\nabla Y\|^{1+\f2N}_{L^{2(1+\f2N)}(\Omega)}\|\nabla Y\|_{L^\i(\Omega)}^{1-\f2N}\\[3mm]
\di \qquad \leq C\sup_{t\in[0,T]}\Big(\|\nabla Y\|_{L^{2}(\Omega)}^{\f2N}\|\nabla  Y\|_{L^\i(\Omega)} ^{1-\f2N}\Big)\|\sqrt{|U^\prime(x_1+m(t))|}\nabla  Y\|_{L^2(\Omega)}\|\Delta  Y\|_{L^2(\Omega)}\\[3mm]
\di \qquad \leq C\sup_{t\in[0,T]}\Big(\|\nabla  Y\|_{L^{2}(\Omega)}^{\f2N}\|\nabla  Y\|_{L^\i(\Omega)} ^{1-\f2N}\Big) \Big[\|\sqrt{|U^\prime(x_1+m(t))|}\nabla Y\|_{L^2(\Omega)}^2+\|\Delta  Y\|_{L^2(\Omega)}^2\Big].
\end{array}
\end{equation*}
Using Sobolev inequality with the assumption \eqref{ass-2}-\eqref{ass-3}, we have
\[
|I_3|\le C\eps_0  \Big(\|\sqrt{|U^\prime(x_1+m(t))|}\nabla Y\|_{L^2(\Omega)}^2+\|\Delta Y\|_{L^2(\Omega)}^2\Big).
\] 
For $I_4$, we use Gagliardo-Nirenberg interpolation to estimate
\begin{align*}
\begin{aligned}
|I_4|&\le \int_\Omega|u-V| |\nabla Y| |U^\prime(x_1+m(t))||\widetilde Y| dx\\
&\leq C\||U^\prime(x_1+m(t))| \widetilde Y\|_{L^2(\Omega)}  \|u-V\|_{L^{\f{2N(N-1)}{N^2-3N+4}}(\Omega)}\||\nabla Y|^{\f2N}\|_{L^{\f{N-1}{N-2}N}(\Omega)}\||\nabla Y|^{1-\f2N}\|_{L^\i(\Omega)}\\
&\le   C\|\sqrt{|U^\prime(x_1+m(t))|} \nabla\widetilde Y\|_{L^2(\Omega)} \|u-V\|^{\f1{N-1}}_{L^2(\Omega)}\|\nabla(u-V)\|^{\f{N-2}{N-1}}_{L^2(\Omega)}\\
&\quad\times\|\nabla Y\|^{\f{N-2}{N(N-1)}}_{L^2(\Omega)}\|\Delta Y\|^{\f1{N-1}}_{L^2(\Omega)}\|\nabla Y\|^{1-\f2N}_{L^\i(\Omega)}\\
&\le   \sup_{t\in[0,T]}\Big(\|u-V\|^{\f1{N-1}}_{L^2(\Omega)}\|\nabla Y\|^{\f{N-2}{N(N-1)}}_{L^2(\Omega)}\|\nabla Y\|_{L^\i(\Omega)} ^{1-\f2N}\Big)\\
& \quad\times \|\sqrt{|U^\prime(x_1+m(t))|} \nabla Y\|_{L^2(\Omega)}\|\nabla(u-V)\|^{\f{N-2}{N-1}}_{L^2(\Omega)} \|\Delta Y\|^{\f{1}{N-1}}_{L^2(\Omega)}.
\end{aligned}
\end{align*}
Using Young inequality and Lemma \ref{lem:cont} with assumptions \eqref{ass-2}, \eqref{ass-3} and \eqref{ass-5}, we have
\begin{align*}
\begin{aligned}
|I_4|\le C\eps_0\Big(\|\sqrt{|U^\prime(x_1+m(t))|} \nabla Y\|_{L^2(\Omega)}^2 +\|\nabla(u-V)\|^{2}_{L^2(\Omega)}+ \|\Delta  Y\|^{2}_{L^2(\Omega)}\Big).
\end{aligned}
\end{align*}
For $I_5$, we use Poincar\'e inequality to estimate
\begin{align*}
\begin{aligned}
|I_5|&\leq C\||U^\prime(x_1+m(t))| \widetilde Y\|_{L^2(\Omega)}\Big(\int_{\bbt^{N-1}}\int_{|x_1+m(t)|\leq M+1} (w_1-h(t)) dx\Big)^{\f12}\\[3mm]
&\leq C \|\sqrt{|U^\prime(x_1+m(t))|} \nabla \widetilde Y\|_{L^2(\Omega)} \|\nabla w_1\|_{L^2(\Omega)}\\[3mm]
& \leq \frac{1}{4}\|\sqrt{|U^\prime(x_1+m(t))|} \nabla Y\|^2_{L^2(\Omega)}+C \|\nabla w_1\|_{L^2(\Omega)}^2.
\end{aligned}
\end{align*}
To estimate $\|\nabla w_1\|_{L^2(\Omega)}$, we notice that since
\beq\label{form-w}
\frac{A(u|V)}{u-V} = (u-V) \int_0^1 \int_0^1 A^{\prime\prime} (V+s\tau(u-V)) \tau ds d\tau,
\eeq
and then
\begin{align*}
\begin{aligned}
&\nabla \frac{A(u|V)}{u-V}=\nabla(u-V)\int_0^1 \int_0^1 A^{\prime\prime} (V+s\tau(u-V)) \tau ds d\tau \\
& \qquad +(u-V)\int_0^1 \int_0^1 A''' (V+s\tau(u-V)) \tau \Big(U'( Y+x_1) (\nabla Y +e_1) +st\nabla (u-V) \Big)  ds d\tau, 
\end{aligned}
\end{align*}
we have
\begin{align*}
\begin{aligned}
\|\nabla w\|_{L^2(\Omega)} &\le C(\|\nabla(u-V)\|_{L^2(\Omega)}+(\|\nabla Y\|_{L^{\infty}}+1)\|(u-V)U'( Y+x_1)\|_{L^2(\Omega)}\\
&\quad+\|(u-V)\nabla(u-V)\|_{L^2(\Omega)}).
\end{aligned}
\end{align*}
We now use the maximum principle 
\beq\label{reg-1}
\|u\|_{L^{\infty}((0,\infty)\times\Omega)}\le \|u_0\|_{L^{\infty}(\Omega)}.
\eeq
Notice that if $u_0-U\in H^s(\Omega)$ with $s>\frac{N}{2}$, and thus $u_0-U\in L^{\infty}(\Omega)$, we have $u_0\in L^{\infty}(\Omega)$ thanks to $U\in L^{\infty}(\Omega)$. 
Thus, using maximum principle \eqref{reg-1} and $V\in L^{\infty}((0,T)\times\Omega)$, we see that
\[
\|(u-V)\nabla(u-V)\|_{L^2(\Omega)} \le C\|\nabla(u-V)\|_{L^2(\Omega)}.
\]
It now remains to estimate $\|(u-V)U'(Y+x_1)\|_{L^2(\Omega)}$. Using \eqref{Y-est-0}, \eqref{ine-1} and Lemma \ref{lem:poincare} with $\phi_1=|U'|$ and $\phi_2=U'$, we have
\begin{equation}
\begin{array}{ll}
\di \|(u-V)U'( Y+x_1)\|^2_{L^2(\Omega)}\\
\di\quad \leq C\Big[ \|(u-V)U'(x_1+m)\|^2_{L^2(\Omega)}+\|(u-V)(U'(Y+x_1)-U^\prime(x_1+m))\|^2_{L^2(\Omega)}\Big]\\
\di \quad \leq C\Big[\|(u-V)U'(x_1+m)\|^2_{L^2(\Omega)}+\|(u-V)\int_0^1U^{\prime\prime}(\theta \widetilde Y+x_1+m)d\theta \widetilde Y\|^2_{L^2(\Omega)}\Big]\\
\di \quad \leq C \|(u-V)\sqrt{|U'(x_1+m)|}\|^2_{L^2(\Omega)}\\
\di \quad \leq C\Big(\int_\Omega U^\prime(x_1+m)(u-V)dx\Big)^2+C\|\nabla(u-V)\|_{L^2(\Omega)}^2.
\end{array}
\end{equation}
Thus
\begin{align*}
\begin{aligned}
|I_5|\leq  \frac{1}{4}\|\sqrt{|U^\prime(x_1+m)|} \nabla Y\|^2_{L^2(\Omega)}+C\Big(\int_\Omega U^\prime(x_1+m)(u-V)dx\Big)^2+C\|\nabla(u-V)\|_{L^2(\Omega)}^2.
\end{aligned}
\end{align*}
Notice that since $\int_\Omega  |U^\prime(x_1+m)|\widetilde Y dx=0$, $I_6=0$.\\
Therefore, combining all estimates above together with Lemma \ref{lem:cont} and assumptions \eqref{ass-2}, \eqref{ass-3} and \eqref{ass-5}, we have that for all $t\in[0,T]$,
\begin{align*}
\begin{aligned}
&\int_\Omega  |U^\prime(x_1+m(t))|\widetilde Y^2dx+\int_0^T\int_\Omega |U^\prime(x_1+m(t))||\nabla  Y|^2 dxdt \\
&\quad\le C \int_0^T\Big(\int_\Omega U^\prime(x_1+m(t))(u-V)dx\Big)^2 dt+C\int_0^T\int_\Omega |\nabla(u-V)|^2dxdt\\
&\qquad+C\eps_0\int_0^T\int_\Omega |\Delta Y|^2dxdt \\
&\quad\le C\eps_0^3.
\end{aligned}
\end{align*}
\end{proof}

The next lemma provides the proof of $L^2$ estimates on $\nabla Y$ in the estimate \eqref{lower-Y-1}.

\begin{lemma}\label{lem:DY}
Under the same assumptions as in Lemma \ref{lem:Y}, there exists a constant $C>0$ such that
\begin{align}
\begin{aligned} \label{DY-est}
\int_{\Omega} |\nabla Y|^2 dx + \int_0^T \int_{\Omega} |\Delta Y|^2 dx ds\leq C\eps_0^3,\quad \forall t\in(0,T].
\end{aligned}
\end{align}
\end{lemma}
\begin{proof}
Multiplying the equation \eqref{Y-eq} by $-\Delta Y$ and integrating the resulting equation over $\Omega$ yield that
\begin{equation}\label{nye}
\begin{array}{ll}
\di \f{d}{dt}\int_\Omega\f12|\nabla Y|^2 dx+\int_\Omega |\Delta Y|^2dx\\[3mm]
\di=-\int_\Omega A^\prime_1(U( Y+x_1))\partial_{x_1} Y \Delta Y dx+\int_\Omega \sum_{i=2}^N A_i^\prime(U(Y+x_1))\partial_{x_i} Y \Delta Y dx\\[3mm]
\di \quad -\int_\Omega A_1^\prime(U(Y+x_1))|\nabla Y|^2\Delta Y dx-\int_\Omega w\cdot\nabla Y \Delta Y dx\\[3mm]
\di \quad -\int_\Omega (w_1-h_M(t))\psi_M(x_1+m(t))\Delta Y dx:=\sum_{i=1}^5 E_i.
\end{array}
\end{equation}
We now estimate the five terms on the right hand side of \eqref{nye}. First, integration by parts implies that
\begin{equation}\label{E1}
\begin{array}{ll}
\di E_1=\int_\Omega A^\prime_1(U( Y+x_1))\partial_{x_1} (\f{|\nabla Y|^2}{2})  dx+\int_\Omega A^{\prime\prime}_1(U( Y+x_1)) U^\prime( Y+x_1) \nabla ( Y+x_1)\cdot \nabla Y  \partial_{x_1}Ydx\\
\di \quad =-\int_\Omega A^{\prime\prime}_1(U( Y+x_1))U^\prime ( Y+x_1)(\partial_{x_1} Y+1)\f{|\nabla Y|^2}{2} dx\\
\di \quad \quad +\int_\Omega A^{\prime\prime}_1(U(
 Y+x_1)) U^\prime(
  Y+x_1) \sum_{i=2}^N(\partial_{x_i} Y)^2  \partial_{x_1}Ydx\\
\di \quad \quad +\int_\Omega A^{\prime\prime}_1(U(
 Y+x_1)) U^\prime(Y+x_1) (\partial_{x_1}Y)^3dx+\int_\Omega A^{\prime\prime}_1(U(
   Y+x_1)) U^\prime(
    Y+x_1) (\partial_{x_1}Y)^2dx\\[3mm]
\di \quad =-\int_\Omega A^{\prime\prime}_1(U(
 Y+x_1)) |U^\prime(
  Y+x_1)| (\partial_{x_1}Y)^2dx\\
\di\qquad -\int_\Omega A^{\prime\prime}_1(U(
 Y+x_1))U^\prime (
  Y+x_1)\partial_{x_1} Y\f{|\nabla Y|^2}{2} dx\\
\di \quad \quad +\int_\Omega A^{\prime\prime}_1(U(
 Y+x_1)) U^\prime(
  Y+x_1) \sum_{i=2}^N(\partial_{x_i} Y)^2  \partial_{x_1}Ydx\\
\di \quad \quad +\int_\Omega A^{\prime\prime}_1(U(
 Y+x_1)) U^\prime(
  Y+x_1) (\partial_{x_1}Y)^3dx-\int_\Omega A^{\prime\prime}_1(U(
   Y+x_1)) U^\prime(
    Y+x_1)\f12 \sum_{i=2}^N(\partial_{x_i}Y)^2 dx\\
\di\quad :=-\int_\Omega A^{\prime\prime}_1(U(
 Y+x_1)) |U^\prime(
  Y+x_1)| (\partial_{x_1}Y)^2dx+\sum_{i=1}^4 E_{1i}.
\end{array}
\end{equation}
Using the same arguments as in previous proofs, we estimate that for each $i=1,2,3$,
\begin{equation}\label{E11-2-3}
\begin{array}{ll}
\di |E_{1i}|\leq  \f18 \int_\Omega A^{\prime\prime}_1(U(
 Y+x_1)) |U^\prime(
  Y+x_1)| (\partial_{x_1}Y)^2dx+C\||\nabla Y|^{1+\f2N}\|^2_{L^2(\Omega)}\||\nabla Y|^{1-\f2N}\|_{L^\infty(\Omega)}^2\\[3mm]
\di  \leq  \f18 \int_\Omega A^{\prime\prime}_1(U( Y+x_1)) |U^\prime( Y+x_1)| (\partial_{x_1}Y)^2dx\\
\di \qquad\qquad\qquad\qquad  +C\sup_{t\in[0,T]}\Big(\|\nabla Y\|^{\f4N}_{L^2(\Omega)}\|\nabla Y\|_{L^\infty(\Omega)}^{2(1-\f2N)}\Big)\int_\Omega |\Delta Y|^2dx,
\end{array}
\end{equation}
and
\begin{equation}\label{E14}
\begin{array}{ll}
\di E_{14}=  - \int_\Omega A^{\prime\prime}_1(U( Y+x_1)) U^\prime(x_1+m(t))\f12 \sum_{i=2}^N(\partial_{x_i}Y)^2dx\\
\di \qquad\qquad   - \int_\Omega A^{\prime\prime}_1(U( Y+x_1)) \big(U^\prime( Y+x_1)-U^\prime(x_1+m(t))\big)\f12 \sum_{i=2}^N(\partial_{x_i}Y)^2dx\\
\di \qquad \leq  C\|\sqrt{|U^\prime(x_1+m(t))|}\nabla Y\|^2_{L^2(\Omega)} \\
\di\qquad\qquad - \int_\Omega A^{\prime\prime}_1(U( Y+x_1)) \int_0^1 U''(\theta\widetilde Y+x_1+m(t))d\theta  \widetilde Y \f12 \sum_{i=2}^N(\partial_{x_i}Y)^2dx\\
\end{array}
\end{equation}
Since 
\[
|U''(\theta \widetilde Y+x_1+m(t))  \widetilde Y| \le C|U'(x_1+m(t))|^{2/3}\eps_0(1+|x_1+m(t)|) \le C\eps_0, 
\]
and for each $i=2,\cdots,N$,
\[
\int_{\bbt^{N-1}} \partial_{x_i}Y dx'=0,
\]
we use Poincar\'e inequality to get
\[
E_{14}\le C\|\sqrt{|U^\prime(x_1+m(t))|}\nabla Y\|^2_{L^2(\Omega)} + C\v_0\int_\Omega |\Delta Y|^2dx.
\]
Similarly, since
\begin{equation}\label{E2}
\begin{array}{ll}
\di E_2 =-\int_\Omega \sum_{i=2}^NA^{\prime}_i(U( Y+x_1))\partial_{x_i} \big(\f{|\nabla Y|^2}{2}\big) dx\\
\di \quad \quad -\int_\Omega \sum_{i=2}^NA^{\prime\prime}_i(U( Y+x_1)) \partial_{x_i}Y U^\prime( Y+x_1) \nabla ( Y+x_1)\cdot \nabla Y  dx\\[3mm]
\di \quad =-\int_\Omega \sum_{i=2}^NA^{\prime\prime}_i(U( Y+x_1)) \partial_{x_i}Y U^\prime( Y+x_1)\f{|\nabla Y|^2}{2} dx\\
\di\qquad -\int_\Omega \sum_{i=2}^NA^{\prime\prime}_i(U( Y+x_1)) \partial_{x_i}Y U^\prime( Y+x_1) \partial_{x_1}Y  dx:=\sum_{i=1}^2 E_{2i},\\
\end{array}
\end{equation}
we estimate
\begin{equation}\label{E21}
\begin{array}{ll}
\di |E_{21}|\leq C\sum_{i=2}^N\|\partial_{x_i} Y\|_{L^2(\Omega)}\||\nabla Y|^{1+\f2N}\|_{L^2(\Omega)}\||\nabla Y|^{1-\f2N}\|_{L^\infty(\Omega)}\\
\di \qquad \leq C\|\nabla^2 Y\|_{L^2(\Omega)}\|\nabla Y\|^{1+\f2N}_{L^{2(1+\f2N)}(\Omega)}\|\nabla Y\|_{L^\infty(\Omega)}^{1-\f2N}\\
\di \qquad \leq C\sup_{t\in[0,T]}\Big(\|\nabla Y\|_{L^2(\Omega)}^{\f2N}\|\nabla Y\|_{L^\infty(\Omega)}^{1-\f2N}\Big) \|\Delta Y\|_{L^2(\Omega)}^2,
\end{array}
\end{equation}
and 
\begin{equation}\label{E22}
\begin{array}{ll}
\di E_{22}=-\int_{\Omega} \sum_{i=2}^NA^{\prime\prime}_i(U( Y+x_1)) \partial_{x_i}Y U^\prime(x_1+m(t)) \partial_{x_1}Y dx\\
\di\qquad\quad -\int_{\Omega} \sum_{i=2}^NA^{\prime\prime}_i(U(  Y+x_1))) \partial_{x_i}Y \big(U^\prime( Y+x_1)-U^\prime(x_1+m(t))\big) \partial_{x_1}Ydx \\
\di \qquad \leq C\|\sqrt{|U^\prime(x_1+m(t))|}\nabla  Y\|^2_{L^2(\Omega)}+C\int_\Omega\Big|\int_0^1 U^{\prime\prime}(\theta  \widetilde Y+x_1+m(t)) d\theta\Big|| \widetilde Y|\sum_{i=2}^N|\partial_{x_i} Y| |\partial_{x_1} Y|dx\\
\di \qquad \leq C\|\sqrt{|U^\prime(x_1+m(t))|}\nabla  Y\|^2_{L^2(\Omega)}+C\eps_0\int_\Omega\sqrt{|U'(x_1+m(t))|}
\sum_{i=2}^N|\partial_{x_i} Y| |\partial_{x_1} Y|dx\\
\di \qquad \leq C\|\sqrt{|U^\prime(x_1+m(t))|}\nabla  Y\|^2_{L^2(\Omega)}+C\eps_0\sum_{i=2}^N\|\partial_{x_i} Y\|_{L^2(\Omega)} \|\sqrt{|U^{\prime}(x_1+m(t))|}\partial_{x_1} Y\|_{L^2(\Omega)}\\
\di \qquad \leq C\|\sqrt{|U^\prime(x_1+m(t))|}\nabla  Y\|^2_{L^2(\Omega)}+C\v_0\sum_{i=2}^N\|\partial_{x_ix_i} Y\|_{L^2(\Omega)} \|\sqrt{|U^{\prime}(x_1+m(t))|}\partial_{x_1} Y\|_{L^2(\Omega)}\\
\di \qquad \leq C\|\sqrt{|U^\prime(x_1+m(t))|}\nabla  Y\|^2_{L^2(\Omega)}+C\v_0\sum_{i=2}^N\|\partial_{x_ix_i} Y\|_{L^2(\Omega)}^2.
\end{array}
\end{equation}
Likewise, we estimate
\begin{equation}\label{E3}
\begin{array}{ll}
\di |E_{3}|\leq C \|\Delta  Y\|_{L^2(\Omega)}\||\nabla Y|^{1+\f2N}\|_{L^2(\Omega)}\||\nabla Y|^{1-\f2N}\|_{L^\i(\Omega)}\\
\di \qquad \leq  \|\Delta  Y\|_{L^2(\Omega)}\|\nabla Y\|_{L^{2(1+\f2N)}(\Omega)}^{1+\f2N}\|\nabla Y\|_{L^\i(\Omega)} ^{1-\f2N}\\
\di \qquad \leq  \sup_{t\in[0,T]}\Big(\|\nabla Y\|_{L^{2}(\Omega)}^{\f2N}\|\nabla Y\|_{L^\i(\Omega)} ^{1-\f2N}\Big) \|\Delta  Y\|^2_{L^2(\Omega)}.
\end{array}
\end{equation}
and
\begin{equation}\label{E4}
\begin{array}{ll}
\di |E_{4}|\leq C\int|u-V||\nabla Y||\Delta Y| dx\\
\di\qquad \leq  \|u-V\|_{L^{\f{2N(N-1)}{N^2-3N+4}}(\Omega)}\|\Delta Y\|_{L^2(\Omega)}\||\nabla Y|^{\f2N}\|_{L^{\f{N-1}{N-2}N}(\Omega)}\||\nabla Y|^{1-\f2N}\|_{L^\i(\Omega)}\\
\di \qquad \leq  \|u-V\|^{\f1{N-1}}_{L^2(\Omega)}\|\nabla(u-V)\|^{\f{N-2}{N-1}}_{L^2(\Omega)}\|\Delta Y\|_{L^2(\Omega)}\|\nabla Y\|^{\f{N-2}{N(N-1)}}_{L^2(\Omega)}\|\Delta Y\|^{\f1{N-1}}_{L^2(\Omega)}\|\nabla Y\|^{1-\f2N}_{L^\i(\Omega)}\\
\di \qquad \leq  \sup_{t\in[0,T]}\Big(\|u-V\|^{\f1{N-1}}_{L^2(\Omega)}\|\nabla Y\|^{\f{N-2}{N(N-1)}}_{L^2(\Omega)}\|\nabla Y\|_{L^\i(\Omega)} ^{1-\f2N}\Big)\|\nabla(u-V)\|^{\f{N-2}{N-1}}_{L^2(\Omega)} \|\Delta  Y\|^{\f{N}{N-1}}_{L^2(\Omega)}\\
\di \qquad \leq  \sup_{t\in[0,T]}\Big(\|u-V\|^{\f1{N-1}}_{L^2(\Omega)}\|\nabla Y\|^{\f{N-2}{N(N-1)}}_{L^2(\Omega)}\|\nabla Y\|_{L^\i(\Omega)} ^{1-\f2N}\Big)\Big[\|\nabla(u-V)\|^{2}_{L^2(\Omega)}+ \|\Delta  Y\|^{2}_{L^2(\Omega)}\Big].
\end{array}
\end{equation}
Using the same estimates as the term $I_4$ in the proof of Lemma \ref{lem:Y}, we have 
\begin{equation}\label{E5}
\begin{array}{ll}
\di |E_{5}|\leq C\|\Delta Y\|_{L^2(\Omega)}\Big(\int_{\bbt^{N-1}}\int_{|x_1+m(t)|\leq M+1} (w_1-h(t)) dx\Big)^{\f12}\\
\di\qquad \leq \f18 \|\Delta Y\|_{L^2(\Omega)}^2+C\Big(\int_\Omega U^\prime(x_1+m(t))(u-V)dx\Big)^2+C\|\nabla(u-V)\|_{L^2(\Omega)}^2.
\end{array}
\end{equation}
Therefore, combining all estimates above together with Lemma \ref{lem:cont}, \ref{lem:Y} and assumptions \eqref{ass-2}, \eqref{ass-3} and \eqref{ass-5}, we have that for all $t\in[0,T]$,
\begin{align*}
\begin{aligned}
&\int_{\Omega} |\nabla Y|^2 dx + \int_0^T \int_{\Omega} |\Delta Y|^2 dx dt \\
&\quad\le C\int_0^T\int_\Omega |U'(x_1+m(t))| |\nabla Y|^2dxdt+C\int_0^T\Big(\int_\Omega U^\prime(x_1+m(t))(u-V)dx\Big)^2 dt\\
&\qquad+C\int_0^T\int_\Omega |\nabla(u-V)|^2dxdt\\
&\quad\le C\eps_0^3.
\end{aligned}
\end{align*}
which completes the proof.
\end{proof}

\subsection{Proof of \eqref{higher-Y-1} in Proposition \ref{prop:priori}}

We first complete the proof of Proposition \ref{prop:priori}. We first recall a priori estimates in Lemma \ref{lem:cont} and Lemma \ref{lem:Y}, \ref{lem:DY} i.e.,
\beq\label{u-1}
\|u-V\|_{L^{\infty}(0,T;L^2(\Omega))}+\|\nabla(u-V)\|_{L^2((0,T)\times\Omega)} \le C\eps_0^{3/2},
\eeq
and
\beq\label{Y-1}
\|\sqrt{|U'(x_1+m)|} (Y-m)\|_{L^{\infty}(0,T; L^2(\Omega))}+\|\nabla Y\|_{L^{\infty}(0,T; L^2(\Omega))}+\|\Delta Y\|_{L^2((0,T)\times\Omega)} \le C\eps_0^{3/2}.
\eeq

In order to complete the proof of Proposition \ref{prop:priori}, we need to show higher-order estimates: 
\[
\|\nabla  Y\|_{L^{\infty}(0,\infty; H^s_{loc}(\Omega))}+ \|\Delta  Y\|_{L^2(0,\infty; H^{s}_{loc}(\Omega))}\le C\eps_0^{3/2},
\]
where the constant $C>0$ depends on $s, N$.\\
For that, we will use the parabolic regularization, which provides a higher regularity estimates: for any fixed $T_*$,
\[
\|\nabla  Y\|_{L^{\infty}(T_*,\infty; H^s_{loc}(\Omega))}+ \|\Delta  Y\|_{L^2(T_*,\infty; H^{s}_{loc}(\Omega))}\le C(T_*)\eps_0^{3/2},
\]
where $C$ is a constant independent of $\eps_0$ if $T_*$ does not depend on $\eps_0$. However, we see that the life span $T_0$ of the local existence in Proposition \ref{prop:local} depends on the size of the above norm of $Y$, according to the proof of Proposition \ref{prop:local}. Therefore, we will get a shaper local-in-time estimate on $Y$ than Proposition \ref{prop:local} up to any fixed time $t_0>0$.

\subsubsection{Local-in-time estimates} We here get a local-in-time estimate.\\
We first get higher-order estimates on $u-V$, which is used in next step. \\
For any $1\le k\le s$, assume that there exists a constant $C>0$ such that
\beq\label{assume-uV}
\|u-V\|_{L^{\infty}(0,t_0; H^{k-1}(\Omega))}+\|\nabla(u-V)\|_{L^2(0,t_0; H^{k-1}(\Omega))} \le C\eps_0^{3/2},
\eeq
and 
\beq\label{assume-Yk}
\|\nabla Y\|_{L^{\infty}(0,t_0; H^{k-1}(\Omega))}+\|\Delta Y\|_{L^2(0,t_0; H^{k-1}(\Omega))} \le C \eps_0^{3/2}.
\eeq
We subtract \eqref{V-eq} from \eqref{eq} to get
\begin{align}
\begin{aligned} \label{u-V}
&\partial_t (u-V) + \sum_{i=1}^N\partial_{x_i} (A_i(u)-A_i(V)) - w\cdot \nabla V -\Delta(u-V)\\
&\quad = -U'(Y+x_1)\Big(w_{1}(1-\psi_M (x_1)) + h_M(t)(1-\psi_M (x_1)) +g(t)  \Big).
\end{aligned}
\end{align}
A simple computation with \eqref{u-V} implies that for all $t\in (0,1)$,
\begin{align*}
\begin{aligned} 
&\frac{1}{2}\frac{d}{dt}\int_{\Omega} |\nabla^{k}(u-V)|^2 dx +
\int_{\Omega} |\nabla^{k+1}(u-V)|^2 dxds \\
&\quad= \int_{\Omega}\nabla^{k+1}(u-V)\nabla^{k-1}\Big(\sum_{i=1}^N \partial_{x_i} (A_i(u)-A_i(V)) + w\cdot \nabla V   \\
&\qquad - U'(Y+x_1)\big(w_{1}(1-\psi_M (x_1)) + h_M(t)(1-\psi_M (x_1))+g(t)  \big)\Big) dx.
\end{aligned}
\end{align*}
Since 
\[
 \partial_{x_i} (A_i(u)-A_i(V))= A'_i(u)\partial_{x_i}(u-V)+(A'_i(u)-A'(V))\partial_{x_i}V,
\]
we rewrite the terms related to the flux as
\begin{align*}
\begin{aligned} 
& \int_{\Omega}\nabla^{k+1}(u-V)\sum_{i=1}^N \nabla^{k-1}\partial_{x_i} (A_i(u)-A_i(V)) dx\\
&\quad= \int_{\Omega}\nabla^{k+1}(u-V) \sum_{i=1}^N \Big[ \nabla^{k-1} \Big(A'_i(u)\partial_{x_i}(u-V)\Big) +\nabla^{k-1} \Big((A'_i(u)-A_i'(V))\partial_{x_i}V\Big) \Big] dx.
\end{aligned}
\end{align*}
Then, using Sobolev inequality, we estimate
\begin{align*}
\begin{aligned} 
&\Big| \int_{\Omega}\nabla^{k+1}(u-V)\sum_{i=1}^N \nabla^{k-1}\partial_{x_i} (A_i(u)-A_i(V)) dx\Big|\\
&\quad=\Big| \int_{\Omega}\nabla^{k+1}(u-V) \sum_{i=1}^N \Big[A'_i(u)\nabla^{k-1}\partial_{x_i}(u-V)+ \sum_{1\le l\le k-1}\binom{k-1}{l}\nabla^{l} A'_i(u)\nabla^{k-1-l}\partial_{x_i}(u-V)\\
&\qquad +\sum_{0\le m\le k-1}\binom{k-1}{m}\nabla^{m}(A'_i(u)-A'(V))\nabla^{k-1-m}\partial_{x_i}V  \Big]  dx\Big| \\
&\quad\le C \|\nabla^{k+1}(u-V)\|_{L^2(\Omega)}\Big[ \|\nabla^{k}(u-V)\|_{L^2(\Omega)}+ \|u-V\|_{H^{k-1}(\Omega)}^{\alpha}\Big(\|\nabla u\|_{H^{s-1}(\Omega)}^{\beta}+\|\nabla Y\|_{H^{k-1}(\Omega)}^{\gamma} \Big)\Big],
\end{aligned}
\end{align*}
where $\alpha, \beta, \gamma \ge1$ are some constants depending on $k$. Since $\nabla u_0\in H^{s-1}(\Omega)$, applying the energy method to \eqref{eq} together with \eqref{reg-1}, we have
\beq\label{reg-2}
\|\nabla u(t)\|_{H^{s-1}(\Omega)}^2 + \int_0^t\|\nabla^2 u(s)\|_{H^{s-1}(\Omega)}^2 ds \le e^{Ct}\|\nabla u_0\|_{H^{s-1}(\Omega)}^2.
\eeq
Moreover, since \eqref{assume-uV} and \eqref{ass-2} yield that for all $t\le t_0$,
\[
\|u-V\|_{H^{k-1}(\Omega)}^{\alpha} \le C\eps_0^{\frac{3\alpha}{2}} \le C\eps_0^{\frac{3}{2}},
\]
and 
\[
\|\nabla Y\|_{H^{k-1}(\Omega)}^{\gamma}\le C,
\]
we have that for all $t\le t_0$, 
\begin{align*}
\begin{aligned} 
&\Big| \int_{\Omega}\nabla^{k+1}(u-V)\sum_{i=1}^N \nabla^{k-1}\partial_{x_i} (A_i(u)-A_i(V)) dx\Big|\\
&\qquad\le\frac{1}{8}\|\nabla^{k+1}(u-V)\|_{L^2(\Omega)}^2 +C \|\nabla^{k}(u-V)\|_{L^2(\Omega)}^2+C\eps_0^3.
\end{aligned}
\end{align*}
Similarly, using \eqref{form-w}, we have
\begin{align*}
\begin{aligned} 
&\Big| \int_{\Omega}\nabla^{k+1}(u-V) \nabla^{k-1}(w\cdot \nabla V) dx\Big|\\
&\quad\le C \|\nabla^{k+1}(u-V)\|_{L^2(\Omega)} \Big[ \|u-V\|_{H^{k-1}(\Omega)}^{\alpha}\Big(\|\nabla u\|_{H^{s-1}(\Omega)}^{\beta}+\|\nabla Y\|_{H^{k-1}(\Omega)}^{\gamma} \Big) \Big]\\
&\quad\le\frac{1}{8}\|\nabla^{k+1}(u-V)\|_{L^2(\Omega)}^2 +C\eps_0^3.
\end{aligned}
\end{align*}
Likewise, we have
\begin{align*}
\begin{aligned} 
&\Big| \int_{\Omega}\nabla^{k+1}(u-V)\nabla^{k-1}\Big(U'(Y+x_1) w_{1}(1-\psi_M (x_1)) \Big)  dx\Big|\\
&\quad\le\frac{1}{8}\|\nabla^{k+1}(u-V)\|_{L^2(\Omega)}^2 +C\eps_0^3.
\end{aligned}
\end{align*}
Moreover, since
\begin{align*}
\begin{aligned} 
&|h_M(t)|\le C_M \|u-V\|_{L^2(\Omega)},\\
&|g(t)|\le C\|u-V\|_{L^2(\Omega)},
\end{aligned}
\end{align*}
we have
\begin{align*}
\begin{aligned} 
&\Big| \int_{\Omega}\nabla^{k+1}(u-V) \nabla^{k-1}\Big(U'(Y+x_1)\big( h_M(t)(1-\psi_M (x_1))+g(t)  \big)\Big)  dx\Big|\\
&\quad\le\frac{1}{8}\|\nabla^{k+1}(u-V)\|_{L^2(\Omega)}^2 +C\eps_0^3.
\end{aligned}
\end{align*}
Therefore we have
\[
\frac{d}{dt}\int_{\Omega} |\nabla^{k}(u-V)|^2 dx +
\int_{\Omega} |\nabla^{k+1}(u-V)|^2 dxds \le C \|\nabla^{k}(u-V)\|_{L^2(\Omega)}^2+C\eps_0^3.
\]
Using \eqref{ass-4} and \eqref{assume-uV}, we have that
\[
\|\nabla^{k}(u-V)\|_{L^{\infty}(0,t_0; L^2(\Omega))}+\|\nabla^{k+1}(u-V)\|_{L^2((0,t_0)\times\Omega))} \le C \eps_0^{3/2},
\]
which together with \eqref{u-1} and \eqref{assume-uV} implies that 
\beq\label{high-V1}
\|u-V\|_{L^{\infty}(0,t_0; H^{s}(\Omega))}+\|\nabla(u-V)\|_{L^2(0,t_0; H^{s}(\Omega))} < C \eps_0^{3/2}.
\eeq

We next estimate $\nabla^{k+1} Y$ as follows. A straightforward computation for \eqref{Y-eq} with $\varphi\equiv 1$ implies that
\begin{align*}
\begin{aligned} 
&\frac{1}{2}\frac{d}{dt}\int_{\Omega} |\nabla^{k+1}Y|^2 dx +
\int_{\Omega} |\nabla^{k+2}Y|^2 dxds \\
&\quad= -\int_{\Omega} \nabla^{k+2}Y \nabla^{k}\Big( A'_1(U(Y+x_1))\partial_{x_1}Y
 -\sum_{i=2}^{N}A'_i(U(Y+x_1))\partial_{x_i}Y    \\
&\qquad   -A'_1(U(Y+x_1))|\nabla Y|^2-w\cdot \nabla Y + (w_{1}-h_M(t)) \psi_M (x_1)\Big) dxds.
\end{aligned}
\end{align*}
We use the same arguments as before, to estimate
\begin{align*}
\begin{aligned} 
&\Big|\int_{\Omega} \nabla^{k+2}Y \nabla^{k}\Big( A'_1(U(Y+x_1))\partial_{x_1}Y\Big)dx \Big|\\
&\quad=\Big|\int_{\Omega} \nabla^{k+2}Y \Big[ A'_1(U(Y+x_1))\nabla^{k}\partial_{x_1}Y+ \sum_{1\le l\le k}\binom{k}{l}\nabla^{l}A'_1(U(Y+x_1))\nabla^{k-l} \partial_{x_1}Y \Big]  dx\Big| \\
&\quad\le C \|\nabla^{k+2}Y\|_{L^2(\Omega)} \Big[ \|\nabla^{k+1}Y\|_{L^2(\Omega)}+ \|\nabla Y\|_{H^{k-1}(\Omega)}^{\alpha}\Big],
\end{aligned}
\end{align*}
where $\alpha\ge1$ is some constant depending on $k$. Thus, it follows from \eqref{assume-Yk} that
\begin{align*}
\begin{aligned} 
\Big|\int_{\Omega} \nabla^{k+2}Y \nabla^{k}\Big( A'_1(U(Y+x_1))\partial_{x_1}Y\Big)dx \Big|\le\frac{1}{8}\|\nabla^{k+2}Y\|_{L^2(\Omega)}^2 +C\|\nabla^{k+1}Y\|_{L^2(\Omega)}^2+C\eps_0^3.
\end{aligned}
\end{align*}
Likewise, we have
\begin{align*}
\begin{aligned} 
&\Big|\int_{\Omega} \nabla^{k+2}Y\sum_{i=2}^{N}\nabla^{k}\Big(A'_i(U(Y+x_1))\partial_{x_i}Y\Big) dx\Big|\\
&\qquad\le\frac{1}{8}\|\nabla^{k+2}Y\|_{L^2(\Omega)}^2 +C\|\nabla^{k+1}Y\|_{L^2(\Omega)}^2+C\eps_0^3,\\
&\Big|\int_{\Omega} \nabla^{k+2}Y \nabla^{k}\Big(A'_1(U(Y+x_1))|\nabla Y|^2\Big) dxds\Big|\\
&\qquad\le\frac{1}{8}\|\nabla^{k+2}Y\|_{L^2(\Omega)}^2 +C\|\nabla^{k+1}Y\|_{L^2(\Omega)}^2+C\eps_0^3.
\end{aligned}
\end{align*}
Using \eqref{form-w} and \eqref{high-V1}, we estimate that for some constants $\alpha, \beta\ge1$,
\begin{align*}
\begin{aligned} 
\Big|\int_{\Omega} \nabla^{k+2}Y \nabla^{k}(w\cdot \nabla Y)  dx\Big|&\le C \|\nabla^{k+2}Y\|_{L^2(\Omega)}  \|u-V\|_{H^{k}(\Omega)}^{\alpha}\|\nabla Y\|_{H^{k-1}(\Omega)}^{\beta} \\
&\le\frac{1}{8}\|\nabla^{k+2}Y\|_{L^2(\Omega)}^2+C\eps_0^3,\\
\end{aligned}
\end{align*}
and
\begin{align*}
\begin{aligned} 
\Big|\int_{\Omega} \nabla^{k+2}Y\nabla^{k}\Big((w_{1}-h_M(t)) \psi_M (x_1)\Big)  dx\Big|&\le C \|\nabla^{k+2}Y\|_{L^2(\Omega)}  \|u-V\|_{H^{k}(\Omega)}\\
&\le\frac{1}{8}\|\nabla^{k+2}Y\|_{L^2(\Omega)}^2+C\eps_0^3.
\end{aligned}
\end{align*}
Therefore, we have 
\begin{align*}
\begin{aligned} 
\frac{d}{dt}\int_{\Omega} |\nabla^{k+1}Y|^2 dx +
\int_{\Omega} |\nabla^{k+2}Y|^2 dxds \le C\|\nabla^{k+1}Y\|_{L^2(\Omega)}^2+C\eps_0^3.
\end{aligned}
\end{align*}
Using \eqref{assume-Yk} and $Y|_{t=0}=0$, we have
\[
\|\nabla^{k+1}Y\|_{L^{\infty}(0,t_0; L^2(\Omega))}+\|\nabla^{k+2}Y\|_{L^2((0,t_0)\times\Omega))} \le C \eps_0^{3/2},
\]
which together with \eqref{Y-1} and \eqref{assume-Yk} implies that
\[
\|\nabla Y\|_{L^{\infty}(0,t_0; H^{s}(\Omega))}+\|\Delta Y\|_{L^2(0,t_0; H^{s}(\Omega))} \le C \eps_0^{3/2}.
\]

\subsubsection{Global-in-time estimates}

In order to complete the proof of Proposition \ref{prop:priori}, we need to show global-in-time estimates: 
\[
\|\nabla  Y\|_{L^{\infty}(t_0,\infty; H^s_{loc}(\Omega))}+ \|\Delta  Y\|_{L^2(t_0,\infty; H^{s}_{loc}(\Omega))}<C\eps_0^{3/2},
\]
where the constant $C>0$ depends on $s, N$. To this end, we use a parabolic regularization.\\

We first get higher-order estimates on $u-V$, which is used in estimates for $Y$. \\
For any $r>0$, we set $Q_r:= (-\frac{1}{r},0)\times \Omega_r$, $\Omega_r:=(-\frac{1}{r},\frac{1}{r})\times\bbt^{N-1}$. Define smooth functions $\phi_r$ satisfying $0\le \phi_r\le 1$ and
\[
\phi_r(t,y)=\left\{ \begin{array}{ll}
          1& \mbox{if $ (t,x)\in Q_r$},\\
        0& \mbox{if $(t,x)\in Q_{r-1}^c$}.\end{array} \right.
\]  
For any $1\le k\le s$, assume that
\beq\label{assume-V}
\|u-V\|_{L^{\infty}(-\frac{1}{k},0; H^{k-1}(\Omega_{k}))}+\|\nabla(u-V)\|_{L^2(-\frac{1}{k},0; H^{k-1}(\Omega_{k}))} < C\eps_0^{3/2}.
\eeq
A simple computation with \eqref{u-V} implies that for all $t\in (-\frac{1}{k},0)$,
\begin{align*}
\begin{aligned} 
&\frac{1}{2}\int_{\Omega_{k}}\phi_{k+1}^2|\nabla^{k}(u-V)|^2 dx +
\int_{-\frac{1}{k}}^t\int_{\Omega_{k}}\phi_{k+1}^2|\nabla^{k+1}(u-V)|^2 dxds \\
&\quad= \int_{-\frac{1}{k}}^t\int_{\Omega_{k}}\Big[\phi_{k+1}\partial_t\phi_{k+1}|\nabla^{k}(u-V)|^2 -2\phi_{k+1}\nabla\phi_{k+1}\nabla^{k}(u-V) \nabla^{k+1}(u-V)\\
&\qquad+ \nabla(\phi_{k+1}^2\nabla^{k}(u-V))\nabla^{k-1}\Big(\sum_{i=1}^N \partial_{x_i} (A_i(u)-A_i(V)) + w\cdot \nabla V   \\
&\qquad - U'(Y+x_1)\big(w_{1}(1-\psi_M (x_1+m(t))) + h_M(t)(1-\psi_M (x_1+m(t)))+g(t)  \big)\Big) \Big] dxds.
\end{aligned}
\end{align*}
The assumption \eqref{assume-V} yields that
\[
\Big|\int_{-\frac{1}{k}}^t\int_{\Omega_{k}}\phi_{k+1}\partial_t\phi_{k+1}|\nabla^{k}(u-V)|^2 dxds\Big|\le C \|\nabla^{k}(u-V)\|_{L^2(Q_{k})}^2
\le C\eps_0^3,
\]
and
\begin{align*}
\begin{aligned} 
&\Big|\int_{-\frac{1}{k}}^t\int_{\Omega_{k}}\phi_{k+1}\nabla\phi_{k+1}\nabla^{k}(u-V) \nabla^{k+1}(u-V) dxds \Big| \\
&\quad\le \frac{1}{8}\int_{-\frac{1}{k}}^t\int_{\Omega_{k}}\phi_{k+1}^2|\nabla^{k+1}(u-V)|^2 dxds+ C\|\nabla^{k}(u-V)\|_{L^2(Q_{k})}^2\\
&\quad\le \frac{1}{8}\int_{-\frac{1}{k}}^t\int_{\Omega_{k}}\phi_{k+1}^2|\nabla^{k+1}(u-V)|^2 dxds +C\eps_0^3,
\end{aligned}
\end{align*}
where the constants $C$ appeared here and below depend on $k$.\\
We use the same arguments as the local-in-time estimates to get
\begin{align*}
\begin{aligned} 
&\Big|\int_{-\frac{1}{k}}^t\int_{\Omega_{k}}\nabla(\phi_{k+1}^2\nabla^{k}(u-V))\sum_{i=1}^N \nabla^{k-1}\partial_{x_i} (A_i(u)-A_i(V)) dxds\Big|\\
&\quad=\Big|\int_{-\frac{1}{k}}^t\int_{\Omega_{k}}\Big[2\phi_{k+1}\nabla\phi_{k+1}\nabla^{k}(u-V)+ \phi_{k+1}^2\nabla^{k+1}(u-V)\Big]\\
&\qquad\cdot\sum_{i=1}^N \Big[ \nabla^{k-1} \Big(A'_i(u)\partial_{x_i}(u-V)\Big) + \nabla^{k-1} \Big((A'_i(u)-A'(V))\partial_{x_i}V\Big) \Big] dxds\\
&\quad \le C\int_{-\frac{1}{k}}^t \Big[ \|\nabla^{k}(u-V)\|_{L^2(\Omega_{k})}+\Big(\int_{\Omega_{k}}\phi_{k+1}^2|\nabla^{k+1}(u-V)|^2 dx\Big)^{1/2}\Big] \\
&\qquad \times\Big[ \|\nabla^{k}(u-V)\|_{L^2(\Omega_k)}+ \|u-V\|_{H^{k-1}(\Omega_k)}^{\alpha}\Big(\|\nabla u\|_{H^{s-1}(\Omega_k)}^{\beta}+\|\nabla Y\|_{H^{k-1}(\Omega_k)}^{\gamma}\Big) \Big] ds,
\end{aligned}
\end{align*}
where $\alpha, \beta, \gamma \ge1$ are some constants depending on $k$. Thanks to \eqref{u-1}, applying the parabolic regularization to \eqref{eq} together with \eqref{reg-1}, we have
\[
u\in L^{\infty}(-\frac{1}{k},0; H^{s}(\Omega_{k})),
\]
which together with \eqref{assume-V} and \eqref{ass-3} implies that
\begin{align*}
\begin{aligned} 
&\Big|\int_{-\frac{1}{k}}^t\int_{\Omega_{k}}\nabla(\phi_{k+1}^2\nabla^{k}(u-V))\sum_{i=1}^N \nabla^{k-1}\partial_{x_i} (A_i(u)-A_i(V)) dxds\Big|\\
&\qquad\le \frac{1}{8}\int_{-\frac{1}{k}}^t\int_{\Omega_{k}}\phi_{k+1}^2|\nabla^{k+1}(u-V)|^2 dxds +C\eps_0^3.
\end{aligned}
\end{align*}
Likewise, we have
\begin{align*}
\begin{aligned} 
&\Big|\int_{-\frac{1}{k}}^t\int_{\Omega_{k}}\nabla(\phi_{k+1}^2\nabla^{k}(u-V)) \nabla^{k-1}(w\cdot \nabla V) dxds\Big|\\
&\quad\le \frac{1}{8}\int_{-\frac{1}{k}}^t\int_{\Omega_{k}}\phi_{k+1}^2|\nabla^{k+1}(u-V)|^2 dxds +C\eps_0^3,
\end{aligned}
\end{align*}
\begin{align*}
\begin{aligned} 
&\Big|\int_{-\frac{1}{k}}^t\int_{\Omega_{k}}\nabla(\phi_{k+1}^2\nabla^{k}(u-V)) \nabla^{k-1}\Big(U'(Y+x_1) w_{1}(1-\psi_M (x_1)) \Big)  dxds\Big|\\
&\quad\le \frac{1}{8}\int_{-\frac{1}{k}}^t\int_{\Omega_{k}}\phi_{k+1}^2|\nabla^{k+1}(u-V)|^2 dxds +C\eps_0^3,
\end{aligned}
\end{align*}
and
\begin{align*}
\begin{aligned} 
&\Big|\int_{-\frac{1}{k}}^t\int_{\Omega_{k}}\nabla(\phi_{k+1}^2\nabla^{k}(u-V)) \nabla^{k-1}\Big(U'(Y+x_1)\big( h_M(t)(1-\psi_M (x_1+m(t)))+g(t)  \big)\Big)  dxds\Big|\\
&\quad\le \frac{1}{8}\int_{-\frac{1}{k}}^t\int_{\Omega_{k}}\phi_{k+1}^2|\nabla^{k+1}(u-V)|^2 dxds +C\eps_0^3.
\end{aligned}
\end{align*}
Therefore, we get
\begin{align*}
\begin{aligned} 
\int_{\Omega_{k}}\phi_{k+1}^2|\nabla^{k}(u-V)|^2 dx +
\int_{-\frac{1}{k}}^t\int_{\Omega_{k}}\phi_{k+1}^2|\nabla^{k+1}(u-V)|^2 dxds <C\eps_0^3.
\end{aligned}
\end{align*}
Hence we have
\[
\|\nabla^{k}(u-V)\|_{L^{\infty}(-\frac{1}{k+1},0; L^2(\Omega_{k+1}))}+\|\nabla^{k+1}(u-V)\|_{L^2((-\frac{1}{k+1},0)\times \Omega_{k+1})} < C\eps_0^{3/2},
\]
which together with \eqref{u-1} and \eqref{assume-V} implies that for all $0\le k\le s$,
\beq\label{high-V}
\|u-V\|_{L^{\infty}(-\frac{1}{k+1},0; H^{k}(\Omega_{k+1}))}+\|\nabla(u-V)\|_{L^2(-\frac{1}{k+1},0; H^{k}(\Omega_{k+1}))} < C\eps_0^{3/2}.
\eeq

We next estimate $\nabla^{k+1} Y$ as follows. Using the same notations and arguments as before, for any $1\le k\le s$, assume that
\beq\label{assume-k}
\|\nabla Y\|_{L^{\infty}(-\frac{1}{k+1},0; H^{k-1}(\Omega_{k+1}))}+\|\Delta Y\|_{L^2(-\frac{1}{k+1},0; H^{k-1}(\Omega_{k+1}))} <C \eps_0^{3/2}.
\eeq
A straightforward computation for \eqref{Y-eq} with $\varphi\equiv 1$ implies that for all $t\in (-\frac{1}{k+1},0)$,
\begin{align*}
\begin{aligned} 
&\frac{1}{2}\int_{\Omega_{k+1}}\phi_{k+2}^2|\nabla^{k+1}Y|^2 dx +
\int_{-\frac{1}{k+1}}^t\int_{\Omega_{k+1}}\phi_{k+2}^2|\nabla^{k+2}Y|^2 dxds \\
&\quad= \int_{-\frac{1}{k+1}}^t\int_{\Omega_{k+1}}\Big[\phi_{k+2}\partial_t\phi_{k+2}|\nabla^{k+1}Y|^2 -2\phi_{k+2}\nabla\phi_{k+2}\nabla^{k+1}Y \nabla^{k+2}Y\\
&\qquad- \nabla(\phi_{k+2}^2\nabla^{k+1}Y) \nabla^{k}\Big( A'_1(U(Y+x_1))\partial_{x_1}Y
 -\sum_{i=2}^{N}A'_i(U(Y+x_1))\partial_{x_i}Y    \\
&\qquad   -A'_1(U(Y+x_1))|\nabla Y|^2-w\cdot \nabla Y + (w_{1}-h_M(t)) \psi_M (x_1+m(t))\Big)\Big] dxds.
\end{aligned}
\end{align*}
We follow the same arguments as in the previous step. Again, every constant $C$ below depends on $k$. \\
The assumption \eqref{assume-k} yields that
\[
\Big|\int_{-\frac{1}{k+1}}^t\int_{\Omega_{k+1}}\phi_{k+2}\partial_t\phi_{k+2}|\nabla^{k+1}Y|^2 dxds\Big|\le C \|\nabla^{k+1}Y\|_{L^2(Q_{k+1})}^2
<C\eps_0^3,
\]
and 
\begin{align*}
\begin{aligned} 
&\Big|\int_{-\frac{1}{k+1}}^t\int_{\Omega_{k+1}}\phi_{k+2}\nabla\phi_{k+2}\nabla^{k+1}Y \nabla^{k+2}Y dxds \Big|\\
&\quad\le \frac{1}{8}\int_{-\frac{1}{k+1}}^t\int_{\Omega_{k+1}}\phi_{k+2}^2|\nabla^{k+2}Y|^2 dxds+C \|\nabla^{k+1}Y\|_{L^2(Q_{k+1})}^2\\
&\quad<\frac{1}{8}\int_{-\frac{1}{k+1}}^t\int_{\Omega_{k+1}}\phi_{k+2}^2|\nabla^{k+2}Y|^2 dxds +C\eps_0^3.
\end{aligned}
\end{align*}
For other terms related to the flux, we use H$\ddot{\mbox{o}}$lder inequality and Sobolev inequality together with \eqref{ass-2}, to get
\begin{align*}
\begin{aligned} 
&\Big|\int_{-\frac{1}{k+1}}^t\int_{\Omega_{k+1}}\nabla(\phi_{k+2}^2\nabla^{k+1}Y) \nabla^{k}\Big( A'_1(U(Y+x_1))\partial_{x_1}Y\Big)dxds\Big|\\
&\quad=\Big|\int_{-\frac{1}{k+1}}^t\int_{\Omega_{k+1}}\Big[2\phi_{k+2}\nabla\phi_{k+2}\nabla^{k+1}Y+ \phi_{k+2}^2\nabla^{k+2}Y\Big]\\
&\qquad\cdot\Big[ A'_1(U(Y+x_1))\nabla^{k}\partial_{x_1}Y+ \sum_{1\le l\le k}\binom{k}{l}\nabla^{l}A'_1(U(Y+x_1))\nabla^{k-l} \partial_{x_1}Y \Big]  dxds\Big| \\
&\quad\le C\int_{-\frac{1}{k+1}}^t \Big[ \|\nabla^{k+1}Y\|_{L^2(\Omega_{k+1})}+\Big(\int_{\Omega_{k+1}}\phi_{k+2}^2|\nabla^{k+2}Y|^2 dx\Big)^{1/2}\Big] \\
&\qquad \cdot\Big[ \|\nabla^{k+1}Y\|_{L^2(\Omega_{k+1})}+ \|\nabla Y\|_{H^{s}(\Omega_{k+1})}^{\alpha}\Big] ds,
\end{aligned}
\end{align*}
where $\alpha\ge1$ is some constant depending on $k$. Thus, it follows from \eqref{assume-k} and \eqref{ass-3} that
\begin{align*}
\begin{aligned} 
&\Big|\int_{-\frac{1}{k+1}}^t\int_{\Omega_{k+1}}\nabla(\phi_{k+2}^2\nabla^{k+1}Y) \nabla^{k}\Big( A'_1(U(Y+x_1))\partial_{x_1}Y\Big)dxds\Big|\\
&\quad<\frac{1}{8}\int_{-\frac{1}{k+1}}^t\int_{\Omega_{k+1}}\phi_{k+2}^2|\nabla^{k+2}Y|^2 dxds +C\eps_0^3.
\end{aligned}
\end{align*}
Similarly, we have
\begin{align*}
\begin{aligned} 
&\Big|\int_{-\frac{1}{k+1}}^t\int_{\Omega_{k+1}}\nabla(\phi_{k+2}^2\nabla^{k+1}Y) \sum_{i=2}^{N}\nabla^{k}\Big(A'_i(U(Y+x_1))\partial_{x_i}Y\Big) dxds\Big|\\
&\qquad<\frac{1}{8}\int_{-\frac{1}{k+1}}^t\int_{\Omega_{k+1}}\phi_{k+2}^2|\nabla^{k+2}Y|^2 dxds +C\eps_0^3,\\
&\Big|\int_{-\frac{1}{k+1}}^t\int_{\Omega_{k+1}}\nabla(\phi_{k+2}^2\nabla^{k+1}Y) \nabla^{k}\Big(A'_1(U(Y+x_1))|\nabla Y|^2\Big) dxds\Big|\\
&\qquad<\frac{1}{8}\int_{-\frac{1}{k+1}}^t\int_{\Omega_{k+1}}\phi_{k+2}^2|\nabla^{k+2}Y|^2 dxds +C\eps_0^3.
\end{aligned}
\end{align*}
Using \eqref{form-w},  \eqref{assume-k} and \eqref{ass-3}, we have
\begin{align*}
\begin{aligned} 
&\Big|\int_{-\frac{1}{k+1}}^t\int_{\Omega_{k+1}}\nabla(\phi_{k+2}^2\nabla^{k+1}Y) \nabla^{k}(w\cdot \nabla Y)  dxds\Big|\\
&\quad\le C\int_{-\frac{1}{k+1}}^t \Big[ \|\nabla^{k+1}Y\|_{L^2(\Omega_{k+1})}+\Big(\int_{\Omega_{k+1}}\phi_{k+2}^2|\nabla^{k+2}Y|^2 dx\Big)^{1/2}\Big]  \|u-V\|_{H^{k}(\Omega_{k+1})}^{\alpha}\|\nabla Y\|_{H^{s}(\Omega_{k+1})} ds\\
&\quad<\frac{1}{8}\int_{-\frac{1}{k+1}}^t\int_{\Omega_{k+1}}\phi_{k+2}^2|\nabla^{k+2}Y|^2 dxds +C\eps_0^3,
\end{aligned}
\end{align*}
and
\begin{align*}
\begin{aligned} 
&\Big|\int_{-\frac{1}{k+1}}^t\int_{\Omega_{k+1}}\nabla(\phi_{k+2}^2\nabla^{k+1}Y) \nabla^{k}\Big((w_{1}-h_M(t)) \psi_M (x_1+m(t))\Big)  dxds\Big|\\
&\quad\le C\int_{-\frac{1}{k+1}}^t \Big[ \|\nabla^{k+1}Y\|_{L^2(\Omega_{k+1})}+\Big(\int_{\Omega_{k+1}}\phi_{k+2}^2|\nabla^{k+2}Y|^2 dx\Big)^{1/2}\Big] \|u-V\|_{H^{k}(\Omega_{k+1})} ds\\
&\quad<\frac{1}{8}\int_{-\frac{1}{k+1}}^t\int_{\Omega_{k+1}}\phi_{k+2}^2|\nabla^{k+2}Y|^2 dxds +C\eps_0^3.
\end{aligned}
\end{align*}
Therefore, we have
\begin{align*}
\begin{aligned} 
\int_{\Omega_{k+1}}\phi_{k+2}^2|\nabla^{k+1}Y|^2 dx +
\int_{-\frac{1}{k+1}}^t\int_{\Omega_{k+1}}\phi_{k+2}^2|\nabla^{k+2}Y|^2 dxds <C\eps_0^3.
\end{aligned}
\end{align*}
Thus,
\[
\|\nabla^{k+1}Y\|_{L^{\infty}(-\frac{1}{k+2},0; L^2(\Omega_{k+2}))}+\|\nabla^{k+2}Y\|_{L^2((-\frac{1}{k+2},0)\times \Omega_{k+2})} < C\eps_0^{3/2},
\]
which together with \eqref{Y-1} and \eqref{assume-k} implies that 
\[
\|\nabla Y\|_{L^{\infty}(-\frac{1}{s+2},0; H^{s}(\Omega_{s+2}))}+\|\Delta Y\|_{L^2(-\frac{1}{s+2},0; H^{s}(\Omega_{s+2}))} < C \eps_0^{3/2}.
\]
This implies that there exists $C>0$ depending only on $s, N$ such that 
\beq\label{global-Y}
 \|\nabla Y\|_{L^{\infty}(t_0,T; H^{s}_{loc}(\Omega))}\le C\eps_0^{3/2}.
\eeq

\subsection{Proof of \eqref{higher-Y-1} in Proposition \ref{prop:priori2}}
First of all, we use the same argument to get local-in-time estimates on $Y$.
 For any fixed $t_0>0$, and $1\le k\le s$, assume that there exists $C>0$ such that
\beq\label{assume-Yk2}
\|\nabla Y\|_{L^{\infty}(0,\frac{t_0}{2}; H^{k-1}(\Omega))}+\|\Delta Y\|_{L^2(0,\frac{t_0}{2}; H^{k-1}(\Omega))} \le C \eps_0^{3/2}.
\eeq
A simple computation with \eqref{Y-eq} implies that 
\begin{align*}
\begin{aligned} 
&\frac{1}{2}\frac{d}{dt}\int_{\Omega} |\nabla^{k+1}Y|^2 dx +
\int_{\Omega} |\nabla^{k+2}Y|^2 dxds \\
&\quad= -\int_{\Omega} \nabla^{k+2}Y \nabla^{k}\Big( A'_1(U(Y+x_1))\partial_{x_1}Y
 -\sum_{i=2}^{N}A'_i(U(Y+x_1))\partial_{x_i}Y    \\
&\qquad   -A'_1(U(Y+x_1))|\nabla Y|^2 \Big) dxds
\end{aligned}
\end{align*}
Notice that $w=0$ for all $t\le\frac{t_0}{2}$ (see \eqref{w}), therefore, we do not need to estimate $u-V$ unlike the proof of Proposition \ref{prop:priori}.\\
Hence, using the same arguments together with \eqref{assume-Yk2} as before, we get
\[
\|\nabla^{k+1}Y\|_{L^{\infty}(0,\frac{t_0}{2}; L^2(\Omega))}+\|\nabla^{k+2}Y\|_{L^2((0,\frac{t_0}{2})\times\Omega))} < C \eps_0^{3/2},
\]
which together with \eqref{Y-1} and \eqref{assume-Yk2} implies that
\[
\|\nabla Y\|_{L^{\infty}(0,\frac{t_0}{2}; H^{s}(\Omega))}+\|\Delta Y\|_{L^2(0,\frac{t_0}{2}; H^{s}(\Omega))} < C \eps_0^{3/2}.
\]
On the other hand, since the initial condition \eqref{ass-5} has been used in the global-in-time estimate \eqref{global-Y}, we have, under the assumption \eqref{ass-5}, the same result as
\[
 \|\nabla Y\|_{L^{\infty}(\frac{t_0}{2},T; H^{s}_{loc}(\Omega))}\le C\eps_0^{3/2}.
\]

\subsection{Proof of Theorem \ref{thm:general}}
\subsubsection{Global-in-time existence of the shift $ Y$ and contraction of the perturbation $u-V$}
First of all, Proposition \ref{prop:local} implies that 
\begin{align*}
\begin{aligned} 
\|\sqrt{|U'(\cdot+m(t))|} (Y-m(t))\|_{L^{\infty}(0,T_0; L^2(\Omega))} &\le \|\sqrt{|U'(\cdot+m(t))|} Y\|_{L^{\infty}(0,T_0; L^2(\Omega))}+C|m(t)|\\
&\le  C\|\sqrt{|U'(\cdot+m(t))|} Y\|_{L^{\infty}(0,T_0; L^2(\Omega))}.
\end{aligned}
\end{align*}
Thanks to Proposition \ref{prop:priori2}, we use continuation argument to conclude that there exists $\delta_0>0$ sufficiently small such that if $\|u-U\|_{L^2(\Omega)}<\delta_0$ and $u_0\in L^{\infty}(\Omega)$, then there exists $C$ depending only on $s, N$ such that 
\begin{align}
\begin{aligned} \label{Y-space}
&\|\sqrt{|U'(\cdot+m(t))|} (Y-m(t))\|_{L^{\infty}(0,\infty; L^2(\Omega))}+\|\sqrt{|U'(\cdot+m(t))|} \nabla Y\|_{L^2((0,\infty)\times\Omega)}\le C\delta_0\\
& \|\nabla  Y\|_{L^{\infty}(0,\infty; L^2(\Omega))}+ \|\Delta  Y\|_{L^2((0,\infty)\times\Omega)}+\|\nabla  Y\|_{L^{\infty}(0,\infty; H^s_{loc}(\Omega))}+ \|\Delta  Y\|_{L^2(0,\infty; H^{s}_{loc}(\Omega))}\le C\delta_0.
\end{aligned}
\end{align}
In particular, since the Sobolev imbedding implies that
\[
\|\nabla  Y\|_{L^{\infty}((0,\infty)\times\Omega)}\le \|\nabla  Y\|_{L^{\infty}(0,\infty; H^s_{loc}(\Omega))},
\]
it follows from Lemma \ref{lem:cont-2} that for all $t\le t_0$, there exists a constant $C_0$ depending $t_0$ such that 
\[
\int_{\Omega} |u(t,x)-V(t,x)|^2 dx \le C_0\int_{\Omega} |u_0(x)-U(x_1)|^2 dx,
\]
and for all $t>t_0$,
\beq\label{cont-inf}
\frac{1}{2}\int_\Omega (u-V)^2dx +\int_0^{\infty}\int_\Omega |\nabla(u-V)|^2 dx dt+\int_0^{\infty}\Big(\int_\Omega(u-V)U^\prime(x_1) dx\Big)^2 dt  \le\frac{1}{2}\int_\Omega (u(t_0,x)-U(x))^2dx.
\eeq
Likewise, thanks to Proposition \ref{prop:priori} and Lemma \ref{lem:cont}, we have the contraction estimate \eqref{thm-claim2} together with \eqref{Y-space}.

\subsubsection{Large-time behavior of the shift $ Y$}
We here use the same notation $\widetilde Y$ as in proof of Lemma \ref{lem:Y} to denote $\widetilde Y=Y-m(t)$. \\
Set
\begin{equation}\label{ft}
f(t):=\int_\Omega |U( Y+x_1)-U(x_1+m(t))|^2dx.
\end{equation}
We want to show that
\begin{equation}\label{lft}
\lim_{t\rightarrow+\infty} f(t)=0.
\end{equation}
To this end, we show that $f$ and $f'$ are both integrable over $[0,\infty)$.\\
First of all, using the same argument as \eqref{Y-est-0}-\eqref{ine-1}, and then Lemma \ref{lem:fund}, we estimate
\begin{align*}
\begin{aligned} 
\int_0^\infty f(t) dt &=\int_0^\infty\int_\Omega \Big|\int_0^1 U^\prime(\theta \widetilde  Y+x_1+m(t))d\theta \Big|^2|\widetilde Y|^2dx dt\\[3mm]
& \leq C\int_0^\infty\int_\Omega |U^\prime(x_1+m(t))||\nabla\widetilde Y|^2dx dt.
\end{aligned}
\end{align*}
Then, \eqref{Y-space} yields 
\beq\label{f-integral}
 \int_0^\infty f(t) dt <\infty.
\eeq
On the other hand, using the same arguments as in the proof of Lemma \ref{lem:Y}, we estimate
\begin{align*}
\begin{aligned} 
&\int_0^{\infty} |f^\prime(t)| dt\\
&=\int_0^{\infty}\Big|\int_\Omega 2\Big(U( Y+x_1)-U(x_1+m(t))\Big)\Big(U^\prime ( Y+x_1)\partial_t  Y-U^\prime(x_1+m(t))m^\prime(t)\Big) dx\Big|dt\\
&=\int_0^{\infty}\Big|\int_\Omega 2\Big(U(Y+x_1)-U(x_1+m(t))\Big)\Big[U^\prime (Y+x_1)\Big(A'_1(U(Y+x_1))\partial_{x_1} Y\\
&\quad  -\sum_{i=2}^{N}A'_i(U(Y+x_1))\partial_{x_i}Y +A'_1(U(Y+x_1))|\nabla_x  Y|^2+w\cdot \nabla_x  Y+\Delta Y \\
&\quad + (w_{1}- h_M(t)) \psi_M (x_1+m(t)) + h_M(t) +g(t)-U^\prime(x_1+m(t))m^\prime(t)\Big]dx\Big|dt\\
& \leq C\int_0^\infty\Big[\|U(Y+x_1)-U(x_1+m(t))\|_{L^2(\Omega)}^2+\|\sqrt{|U^\prime(x_1+m(t))|}\nabla Y\|_{L^2(\Omega)}^2+\|\Delta Y\|^2_{L^2(\Omega)}\\
&\quad +\|\nabla (u-V)\|^2_{L^2(\Omega)}+\Big(\int_\Omega U^\prime(x_1+m(t))(u-V)dx\Big)^2+|h_M(t)|^2+|g(t)|^2+|m^\prime(t)|^2\Big] dt.
\end{aligned}
\end{align*}
Then, we use \eqref{Y-space}, \eqref{cont-inf}, \eqref{f-integral}, \eqref{K6} and \eqref{mp} to get
\begin{align*}
\begin{aligned} 
\int_0^{\infty} |f^\prime(t)| dt \leq C.
\end{aligned}
\end{align*}
Therefore, $f$ and $f'$ are both integrable over $[0,\infty)$, which completes \eqref{lft}.

\begin{appendix}
\setcounter{equation}{0}
\section{Proof of Proposition \ref{prop:local}}
\subsection{Local existence of Eq. \eqref{Y-eq} with $\varphi\equiv1$}
First of all, we construct approximate solutions $(Y_n)_{n\ge 0}$, following iteration scheme:\\
Set
\[ Y_0(t,x)=0,\quad t\ge 0,~ x\in\Omega. \]
Then, for a given $n$-th approximate solution $Y_n$, we define $Y_{n+1}$ as a solution of the linear equation
\begin{align}
\begin{aligned} \label{Y-n}
&\partial_t Y_{n+1} -A'_1(U(Y_n+x_1))\partial_{x_1}Y_{n+1} +\sum_{i=2}^{N}A'_i(U(Y_n+x_1))\partial_{x_i}Y_{n+1} -A'_1(U(Y_n+x_1))|\nabla Y_n|^2\\&\quad  +w_n\cdot \nabla Y_{n}-\Delta Y_{n+1} = -w_{n,1}\psi_M (x_1+m_n)  - h_{n,M}(t)(1-\psi_M (x_1+m_n)) - g_n(t),
\end{aligned}
\end{align}
where the notations $w_n, w_{n,1}, h_{n,M}$ and $m_n$ mean that $Y_n$ replaces $Y$ in those functions $w, w_{1}, h_{M}$ and $m$, respectively, appeared in the Eq. \eqref{Y-eq} with $\varphi\equiv1$.\\
We will show that for any $R>0$, there exists $T_0>0$ such that
\beq\label{uniform}
\|\sqrt{|U'(\cdot + m_{n})|} Y_n\|_{L^{\infty}(0,T_0; L^2(\Omega))}+ \|\nabla Y_n\|_{L^{\infty}(0,T_0; H^s(\Omega))}+ \|\Delta Y_n\|_{L^2(0,T_0; H^{s}(\Omega))} \le R
\eeq
For notational simplification, we rewrite \eqref{Y-n} into a linear equation:
\begin{align}
\begin{aligned} \label{Y-linear}
&\partial_t Y -A'_1(U(Z+x_1))\partial_{x_1}Y +\sum_{i=2}^{N}A'_i(U(Z+x_1))\partial_{x_i}Y\\
& -A'_1(U(Z+x_1))|\nabla Z|^2 +w_{Z}\cdot \nabla Z-\Delta Y\\
&\quad = -w_{Z,1}\psi_M (x_1+m_Z)  - h_{Z,M}(t)(1-\psi_M (x_1+m_Z)) - g_Z(t),\\
&Y|_{t=0}=0,
\end{aligned}
\end{align}
where the notations $w_Z, w_{Z,1}, h_{Z,M}$ and $m_Z$ mean that $Z$ replaces $Y$ in those function $w, w_{1}, h_{M}$ and $m$, respectively, appeared in the Eq. \eqref{Y-eq} with $\varphi\equiv1$.\\

Assume that for any $R>0$, there exists $T_0>0$ such that
\begin{equation} \label{assume-Z}
\|\sqrt{|U'(\cdot + m_Z)|} Z\|_{L^{\infty}(0,T_0; L^2(\Omega))}+ \|\nabla Z\|_{L^{\infty}(0,T_0; H^s(\Omega))}+ \|\Delta Z\|_{L^2(0,T_0; H^{s}(\Omega))} \le R.
\end{equation}
We first estimate $\|\nabla Y\|_{L^{\infty}(0,T_0; H^s(\Omega))}+ \|\Delta Y\|_{L^2(0,T_0; H^{s}(\Omega))} \le R$.\\
For any $k$ with $0\le k\le s$, it follows from \eqref{Y-linear} that for all $t\in (0,T_0)$,
\begin{align*}
\begin{aligned} 
&\frac{1}{2}\frac{d}{dt}\int_{\Omega} |\nabla^{k+1}Y|^2 dx +
\int_{\Omega} |\nabla^{k+2}Y|^2 dxds \\
&\quad= -\int_{\Omega} \nabla^{k+2}Y \nabla^{k}\Big( A'_1(U(Z+x_1))\partial_{x_1}Y\Big)
 + \nabla^{k+2}Y \nabla^{k}\Big(\sum_{i=2}^{N}A'_i(U(Z+x_1))\partial_{x_i}Y \Big)   \\
&\qquad   - \nabla^{k+2}Y \nabla^{k}\Big(A'_1(U(Z+x_1))|\nabla Z|^2\Big)+ \nabla^{k+2}Y \nabla^{k}(w_{Z}\cdot \nabla Z)\\
&\qquad +\nabla^{k+2}Y \nabla^{k} \Big(w_{Z,1}\psi_M (x_1+m_Z)\Big) -\nabla^{k+2}Y  \nabla^{k}\psi_M (x_1+m_Z) h_{Z,M}(t)dx\\
&:=\sum_{i=1}^6 I_i.
\end{aligned}
\end{align*}
For terms related to the flux, we use H$\ddot{\mbox{o}}$lder inequality and Sobolev inequality together with \eqref{assume-Z}, to get
\begin{align*}
\begin{aligned} 
|I_1|&=\Big|\int_{\Omega} \nabla^{k+2}Y \Big[ A'_1(U(Z+x_1))\nabla^{k}\partial_{x_1}Y+ \sum_{1\le l\le k}\binom{k}{l}\nabla^{l}A'_1(U(Z+x_1))\nabla^{k-l} \partial_{x_1}Y \Big]  dx\Big| \\
&\le C \|\nabla^{k+2}Y\|_{L^2(\Omega)} \Big[ \|\nabla^{k+1}Y\|_{L^2(\Omega)}+ \|\nabla Z\|_{H^{s}(\Omega)}^{\alpha}\|\nabla Y\|_{H^{s}(\Omega)}\Big]\\
&\le C \|\nabla^{k+2}Y\|_{L^2(\Omega)} \Big[ \|\nabla^{k+1}Y\|_{L^2(\Omega)}+ R^{\alpha}\|\nabla Y\|_{H^{s}(\Omega)}\Big]\\
& \le \frac{1}{8}\|\nabla^{k+2}Y\|_{L^2(\Omega)}^2 +C\|\nabla^{k+1}Y\|_{L^2(\Omega)}^2+CR^{2\alpha}\|\nabla Y\|_{H^{s}(\Omega)}^2.
\end{aligned}
\end{align*}
where $\alpha\ge1$ is some constant depending on $k$. \\
Likewise, we have
\begin{align*}
\begin{aligned} 
|I_2|&\leq\frac{1}{8}\|\nabla^{k+2}Y\|_{L^2(\Omega)}^2 +C\|\nabla^{k+1}Y\|_{L^2(\Omega)}^2+CR^{2\alpha}\|\nabla Y\|_{H^{s}(\Omega)}^2,\\
|I_3|&\leq\frac{1}{8}\|\nabla^{k+2}Y\|_{L^2(\Omega)}^2 +C\|\nabla^{k+1}Y\|_{L^2(\Omega)}^2+CR^{2\beta},
\end{aligned}
\end{align*}
where $\alpha, \beta\ge1$ are some constants depending on $k$.\\
To estimate $I_4$, notice that
\[
\|\nabla u\|_{L^{\infty}(0,T_0;H^{s-1}(\Omega))}\le e^{CT_0}\|\nabla u_0\|_{H^{s-1}(\Omega)},\quad\mbox{and}\quad \|u\|_{L^{\infty}}\le \|u_0\|_{L^{\infty}}, 
\]
which yield that
\begin{align*}
\begin{aligned} 
|I_4|&\le C \|\nabla^{k+2}Y\|_{L^2(\Omega)} \| \nabla^{k}(w_{Z}\cdot \nabla Z) \|_{L^2(\Omega)} \\
&\le C  \|\nabla^{k+2}Y\|_{L^2(\Omega)} (\|\nabla u\|_{H^{s-1}(\Omega)} + \|\nabla Z\|_{H^{s}(\Omega)}+1)\|\nabla Z\|_{H^s(\Omega)} \\
&\leq\frac{1}{8}\|\nabla^{k+2}Y\|_{L^2(\Omega)}^2+CR^2(e^{CT_0}+R^2+1),\\
\end{aligned}
\end{align*}
and
\begin{align*}
\begin{aligned} 
|I_5|&\le C \|\nabla^{k+2}Y\|_{L^2(\Omega)} \|\nabla^k (w_{Z,1}\psi_M (x_1+m_Z)) \|_{L^2(\Omega)} \\
&\le C  \|\nabla^{k+2}Y\|_{L^2(\Omega)} (\|\nabla u\|_{H^{s-1}(\Omega)} + \|\nabla Z\|_{H^{s}(\Omega)}^{\alpha}+1)\\
&\leq\frac{1}{8}\|\nabla^{k+2}Y\|_{L^2(\Omega)}^2+C(e^{CT_0}+R^{\alpha}+1).
\end{aligned}
\end{align*}
Finally, it follows from \eqref{assume-Z} that
\[
|m_Z|\le C\|\sqrt{|U'(\cdot + m_Z)|} Z\|_{L^{\infty}(0,T_0; L^2(\Omega))} \le CR,
\]
which yields
\begin{align*}
\begin{aligned} 
|I_6|\le C | h_{Z,M}(t)| \|\nabla^{k+2}Y\|_{L^2(\Omega)} \| \nabla^{k}\psi_M (\cdot+m_Z) \|_{L^2(\Omega)} \leq \frac{1}{8}\|\nabla^{k+2}Y\|_{L^2(\Omega)}^2+CR^2.
\end{aligned}
\end{align*}
Therefore, we have 
\begin{align*}
\begin{aligned} 
\frac{d}{dt}\int_{\Omega} |\nabla^{k+1}Y|^2 dx +
\int_{\Omega} |\nabla^{k+2}Y|^2 dxds \leq C\|\nabla^{k+1}Y\|_{L^2(\Omega)}^2+C\|\nabla Y\|_{H^{s}(\Omega)}^2+C_R,
\end{aligned}
\end{align*}
where $C_R$ is a constant depending on $R$.\\
Then, summing the above estimates over $0\le k\le s$, we have
\[
\frac{d}{dt}\|\nabla Y\|_{H^{s}(\Omega)}^2 +
\int_{\Omega} \|\nabla^2 Y\|_{H^{s}(\Omega)}^2 ds \leq C\|\nabla Y\|_{H^{s}(\Omega)}^2+C_R,
\]
which implies
\[
 \|\nabla Y \|_{L^{\infty}(0,T_0; H^s(\Omega))}^2+\|\Delta Y\|_{L^2(0,T_0; H^s(\Omega))}^2 \le C_RT_0e^{CT_0}.
\]
Hence we take $T_0$ to be small so that 
\beq\label{high-Y}
 \|\nabla Y \|_{L^{\infty}(0,T_0; H^s(\Omega))}+\|\Delta Y\|_{L^2(0,T_0; H^s(\Omega))} \le R.
\eeq

We now estimate $\|\sqrt{|U'(\cdot + m)|} Y\|_{L^{\infty}(0,T_0; L^2(\Omega))}\le R$ using the above estimates \eqref{high-Y}.\\
Multiplying \eqref{Y-linear} by $|U'(x_1+m(t))| Y$, and using the same arguments as the two terms $J_1$ and $J_3$ in \eqref{LL0}, we have that
\[
\begin{array}{ll}
\di \partial _t \Big(|U'(x_1+m(t))|\f{ Y^2}{2}\Big)+U^{\prime\prime}(x_1+m(t))m^\prime (t)\f{Y^2}{2}\\
\di 
-\Big[A_1'(U( Z+x_1))-A_1'(U(x_1+m(t)))\Big]|U'(x_1+m(t))| Y\partial_{x_1} Y\\
\di+\sum_{i=2}^N A_i^\prime (U( Z+x_1)) |U'(x_1+m(t))|Y\partial_{x_i} Y-A_1^\prime (U( Z+x_1))|\nabla  Z|^2|U'(x_1+m(t))| Y\\
\di + w_Z\cdot \nabla Z|U'(x_1+m(t))| Y -{\rm div}(|U'(x_1+m(t))|  Y\nabla Y)\\
\di +\partial_{x_1}(\partial_{x_1}|U'(x_1+m(t))|\f{Y^2}{2}) +|U'(x_1+m(t))||\nabla Y|^2\\
\di=-\Big(w_{Z,1}\psi_M (x_1+m_Z)  + h_{Z,M}(t)(1-\psi_M (x_1+m_Z)) + g_Z(t)) \Big)|U'(x_1+m(t))| Y.
\end{array}
\]
Integrating the above equation over $\Omega$, we have
 \begin{align*}
\begin{aligned} 
&\frac{d}{dt}\int_{\Omega} |U^\prime(x_1+m(t))|\f{ Y^2}{2} dx +
\int_{\Omega} |U^\prime(x_1+m(t))||\nabla Y|^2 dx =-\int_\Omega U^{\prime\prime}(x_1+m(t))m^\prime (t)\f{Y^2}{2}dx\\
&\qquad+\int_{\Omega}\Big(A_1'(U(Z+x_1))-A_1'(U(x_1+m(t)))\Big)|U^\prime(x_1+m(t))| Y\partial_{x_1}Y\\
&\qquad -\int_{\Omega}\sum_{i=2}^N A_i^\prime (U( Z+x_1)) |U'(x_1+m(t))|Y\partial_{x_i} Y dx\\
&\qquad   +\int_{\Omega} A'_1(U(Z+x_1))|\nabla Z|^2|U^\prime(x_1+m(t))| Ydx-\int_{\Omega}w_{Z}\cdot \nabla Z|U^\prime(x_1+m(t))| Y dx \\
&\qquad -\int_{\Omega}  \Big(w_{Z,1}\psi_M (x_1+m(t))  + h_{Z,M}(t)(1-\psi_M (x_1+m(t))) +g_Z(t) \Big)|U^\prime(x_1+m(t))| Y dx.
\end{aligned}
\end{align*}
In order to control $m'(t)$, we use the same computations as in Remark \ref{mt-1} and \eqref{m-est}, together with $\|\nabla Y\|_{L^{\infty}((0,T_0)\times\Omega)}\le R$ by \eqref{high-Y}. Then, we have that for all $t\in(0,T_0)$,
 \begin{align*}
\begin{aligned} 
 |m^\prime(t)|&\le C\Big[ \int_\Omega |U^\prime(x_1+m(t))| \Big(|A'_1(U(Z+x_1))||\partial_{x_1} Y|+  \sum_{i=1}^{N}|A_i^\prime(U(Z+x_1))||\partial_{x_i} Y|\Big)dx\\
& \quad +\int_\Omega |U^\prime(x_1+m(t))| |A_1^\prime(U(Z+x_1))||\nabla Z|^2dx +\int_\Omega |U^\prime(x_1+m(t))| |w_Z||\nabla Z|dx\\
& \quad +\int_\Omega |U^\prime(x_1+m(t))| |\Delta Y| dx+\int_\Omega|U^\prime(x_1+m(t))| (|w_{Z,1}| + |h_{Z,M}(t)| + |g_Z(t)|) dx \Big].
\end{aligned}
\end{align*}
Since 
\beq\label{ten-1}
|w_Z|\le C|u-U(Z+x_1)|\le C(\|u_0\|_{L^{\infty}}+\|U\|_{L^{\infty}}),\quad \mbox{and}\quad |h_{Z,M}| + |g_Z|\le C\|w_{Z}\|_{L^{\infty}},
\eeq 
we use \eqref{assume-Z} and \eqref{high-Y} to estimate
\begin{align*}
\begin{aligned} 
 |m^\prime(t)|\le C(R+R^2) \|U'\|_{L^1(\Omega)} + C\|\Delta Y\|_{L^2(\Omega)} \|U'\|_{L^2(\Omega)}\le C(R+R^2),
\end{aligned}
\end{align*}
which yields
\[
\Big| \int_\Omega U^{\prime\prime}(x_1+m(t))m^\prime (t)\f{Y^2}{2}dx \Big| \le C(R+R^2) \int_{\Omega} |U^\prime(x_1+m(t))| Y^2 dx.
\]
Then, we use \eqref{assume-Z} and \eqref{ten-1} to estimate
 \begin{align*}
\begin{aligned} 
\frac{d}{dt}\int_{\Omega} |U^\prime(x_1+m(t))| Y^2 dx +
\int_{\Omega} |U^\prime(x_1+m(t))||\nabla Y|^2 dx  \le C\int_{\Omega} |U^\prime(x_1+m(t))| Y^2 dx + C_R.
\end{aligned}
\end{align*}
which gives
\[
\|\sqrt{|U'(\cdot + m)|} Y\|_{L^{\infty}(0,T_0; L^2(\Omega))}\le \sqrt{C_RT_0e^{CT_0}}\le R, \quad {\rm if}~~T_0\ll1.
\]
Hence, we have shown that the sequence of approximate solutions $(Y_n)_{n\ge 0}$ is uniformly bounded as \eqref{uniform}. 
The remaining part is quite standard, so we only provide a sketch of the proof. Using the uniform estimates \eqref{uniform} and same energy estimates as above, we easily have the strong convergence of sequence $(Y_n)_{n\ge 0}$ towards a limit function $Y$ in a lower-order space $L^{\infty}(0,T_0;L^2(\Omega))\cap L^{2}(0,T_0;H^1(\Omega))$. Then, it is obvious that the limit $Y$ is a solution of \eqref{Y-eq}, and satisfies the estimates \eqref{Y-local-sol}.

\subsection{Local existence of Eq. \eqref{Y-eq}}
For the local existence of Eq. \eqref{Y-eq} in the time interval of $(0,\frac{t_0}{2}]$, we just need the condition $u_0\in L^{\infty}(\Omega)$ without $\nabla u_0\in H^{s-1}(\Omega)$, because \eqref{Y-eq} has no terms related to $w$ and $h_M$ for such a time interval $(0,\frac{t_0}{2}]$, the three terms $I_4$, $I_5$ and $I_6$ in Section A.1 above do not appear. 
\end{appendix}

\end{document}